\documentclass[smallextended,numbook,runningheads]{svjour3}     
\smartqed  

\usepackage{amsmath,amssymb,amsxtra,comment,graphicx,psfrag}
\usepackage{bm,mathrsfs}
\usepackage{mathtools}
\usepackage{graphicx}
\usepackage{epstopdf} 
\usepackage{bbding}
\usepackage{mathptmx}
\usepackage{algorithm}
\usepackage{algorithmic}
\usepackage{color}
\usepackage{cases} 
\usepackage{graphicx,ifthen,latexsym,mathrsfs,multicol}
\DeclareMathOperator\Real {Re}
\def\Re {{\mathbb R}}
\def\N{{\mathbb N}}
\begin{document}

\title{Backward difference formula: The energy technique for subdiffusion equation
\thanks{This work was supported by NSFC 11601206 and Hong Kong RGC grant (No. 25300818).
}}


\author{Minghua Chen         \and
        Fan Yu          \and
       Zhi Zhou
}


\institute{M. Chen (\Envelope) \and F. Yu   \at
              School of Mathematics and Statistics, Gansu Key Laboratory of Applied Mathematics and Complex Systems, Lanzhou University, Lanzhou 730000, P.R. China\\
email:chenmh@lzu.edu.cn;  yuf17@lzu.edu.cn\\
           \and
         Z. Zhou  \at
              Department of Applied Mathematics, The Hong Kong Polytechnic University, Kowloon, Hong Kong, P.R. China\\
              email: zhizhou@polyu.edu.hk
}


\maketitle

\begin{abstract}
Based on the  equivalence of A-stability and G-stability, the energy technique of the six-step BDF method for the  heat equation has been discussed in
[Akrivis,  Chen, Yu,  Zhou,   Math. Comp.,  Revised].
Unfortunately, this theory  is hard to extend the time-fractional PDEs.
In this work, we consider three types of subdiffusion models, namely single-term, multi-term and distributed order fractional diffusion equations.  We present a novel and concise stability   analysis of time stepping schemes
generated by $k$-step backward difference formula (BDF$k$), for approximately
solving the subdiffusion equation.
The analysis mainly relies on the energy technique  by applying Grenander-Szeg\"{o} theorem.
This kind of argument has been widely used to confirm the stability of various $A$-stable schemes (e.g., $k=1,2$).
However, it is not an easy task  for the higher-order BDF methods, due to the loss the $A$-stability.
The core object of this paper is to fill in this gap.
\end{abstract}

\section{Introduction}
Let $T >0, \varrho\in H,$ and consider the   single-term subdiffusion, for which the
governing equation is given by  \cite{CB:11,Carmi:10,ChenD:15,CD:18,SC:20}
\begin{equation} \label{1.1}
\left\{ \begin{array}
 {l@{\quad} l}
\partial_t^{\alpha,\sigma} \left(u(t)-e^{-\sigma t}\varrho\right)+Au(t)=0,~~0<t<T,\\
   u(0)=\varrho,
 \end{array}
 \right.
\end{equation}
where $\sigma\geqslant0$, $A$ is a positive definite, selfadjoint, linear operator on a Hilbert space $(H, (\cdot , \cdot )) $ with domain  $D(A)$ dense in $H$. Let $| \cdot |$  denote the norm on $H$ induced by the inner product $(\cdot , \cdot )$, and introduce on $V, V:=D(A^{1/2}),$
the norm $\| \cdot \|$   by $\| v\| :=| A^{1/2} v |.$
We identify $H$ with its dual, and denote by $V'$ the dual of $V$,
and by $\| \cdot \|_\star$ the dual norm on $V', \|v \|_\star=| A^{-1/2} v |.$
We shall use the notation $(\cdot , \cdot )$ also for the antiduality
pairing between $V$ and $V'$.
Here $\partial_t^{\alpha,\sigma}$, with $\alpha\in(0,1)$, denotes the  fractional substantial derivative
 in time variable \cite{Chendeng:13,ChenD:15,FJBE:06}
\begin{equation}\label{1.2}
\partial_t^{\alpha,\sigma} u(t)=\frac{1}{\Gamma(1-\alpha)}  \left[\frac{\partial}{\partial t}+\sigma\right] \int_{0}^t (t-s)^{-\alpha} e^{-\sigma(t-s)} u(s)ds.
\end{equation}
In addition, under the initial condition $u(0)=\varrho$, the  fractional substantial derivative $\partial_t^{\alpha,\sigma} \left(u(t)-e^{-\sigma t}\varrho\right)$
in the model \eqref{1.1} is identical with the usual Caputo  fractional substantial derivative.
The time-fractional diffusion model \eqref{1.1} could be derived by using the continuous time random walk, describing anomalous diffusion process \cite{Klafter:11}.
They have recently attracted a lot of attention in physics and chemistry (diffusion in different media, reactions, mixing in hydrodynamic flows), biology (from the motion of animals to that of subcellular structures in the crowded environment inside cells), and many other disciplines \cite{Klafter:11,Podlubny:99}.

Our purpose in this paper is to discuss the stability for the more general initial boundary value problem
\begin{equation} \label{1.3}
\left\{ \begin{array}
 {l@{\quad} l}
P\left(\partial_t\right)\left(u(t)-e^{-\sigma t}\varrho\right)+Au(t)=0,~~0<t<T,\\
   u(0)=\varrho,
 \end{array}
 \right.
\end{equation}
where $P\left(\partial_t\right)u $ denotes a   fractional substantial  differential operator of the form
\begin{equation}\label{1.4}
P\left(\partial_t\right)u(t)=\int_0^1\partial_t^{\alpha,\sigma} u(t)d \nu \left(\alpha\right),
\end{equation}
with $\nu \left(\alpha\right)$ a positive measure on $[0,1]$. We remark that it is reduced to fractional order differential operators if $\sigma=0$ in \cite{JLTZ:17}.
These more general models, including in
addition to the single-term model \eqref{1.1}, multi-term and distributed order models are reviewed briefly.
In the multi-term model,
\begin{equation*}
P\left(\partial_t\right)u(t)=\sum_{i=1}^m b_i\partial_t^{\alpha_i,\sigma}u(t),
\end{equation*}
where the constants $b_i$ are positive and $0<\alpha_m<\ldots<\alpha_1<1$. It becomes  the single-term subdiffusion
model \eqref{1.1} when $m = 1$. In the distributed order model,
\begin{equation*}
P\left(\partial_t\right)u(t)=\int_0^1\partial_t^{\alpha,\sigma} u(t) \mu\left(\alpha\right)d \alpha ,
\end{equation*}
where   $\mu\left(\alpha\right)$ is a nonnegative weight function.
Formally, the multi-term subdiffusion model can be seen as the distributed order subdiffusion
model associated with the weight function $\mu\left(\alpha\right)=\sum_{i=1}^m b_i\delta\left(\alpha-\alpha_i\right)$, where $\delta$ is
the Dirac-delta function.

Stability of the A-stable one- and/or two-step BDF methods
can be easily proved  by the energy method. The powerful
 Nevanlinna--Odeh multiplier technique \cite{NO:81} extends the applicability
 of the energy method to the non A-stable three-, four- and five-step BDF methods.
Using results from Dahlquist's G-stability theory \cite{Dahlquist:78}, Liu constructs  the new telescope formulas for the BDF$k$ ($k=3,4,5$) schemes of parabolic equations \cite{Liu:13}, but it fails to BDF6.
According to  G-stability theory and Nevanlinna-Odeh multipliers \cite{NO:81}, Lubich et al. analyze the BDF methods up to order five for parabolic equations \cite{LMV:13}.
Based  on  the cerebrated equivalence of A-stability and G-stability \cite{BC:89,HW:10},
Akrivis et al. construct new multipliers and  analyze the BDF method for parabolic equations \cite{Akrivis:18,AK:16}.
Recently, the energy technique for the six-step BDF method is first established for the heat equations  \cite{ACYZ:20}, which also shows that no Nevanlinna-Odeh multiplier exists.

However, as the mentioned above, the   G-stability theory \cite{Liu:13} or the  equivalence of A-stability and G-stability  \cite{ACYZ:20} are not easy to extend to subdiffusion equation.
Using fractional backward difference formula, error analysis of up to sixth order temporal accuracy for fractional ordinary differential equation has been discussed \cite{Chendeng:13,Lubich:86} with the starting quadrature weights schemes. A few years later, based on operational calculus with sectorial operator,
 nonsmooth data error estimates for fractional evolution equations have been studied in \cite{Cuesta:06,Lu:96} and extended to \cite{Jin:16,Jin:17} including fractional substantial PDEs \cite{SC:20} to restore up to six-order.
Under the time regularity assumption, high order finite difference method  (BDF2) for the anomalous diffusion equation has been studied in \cite{LiD:13} by analyzing the properties of the coefficients.
Using  Grenander-Szeg\"{o} theorem, stability and convergence for time-fractional subdiffusion equation have been provided in \cite{Gao:15,Ji:15} and
also developed in \cite{CD:18} for  BDF2.
To the best of our knowledge, we are unaware of any other published works on stability analysis of BDF$k$ $(k\geqslant3)$ schemes for subdiffusion equation by the energy technique. This gap in the research literature is the motivation for our work. In this paper, we introduce multipliers satisfying the positivity property \eqref{pos-prop} and the A-stability property \eqref{A} for the BDF$k$ $(k\geqslant3)$ method and establish  a novel stability analysis for time-fractional subdiffusion equation by the energy technique.

An outline of the paper is as follows. In Section \ref{sec:1}, we recall the  BDF$k$ (corrected) schemes  for the model  \eqref{1.1} and introduce  multipliers.
In Section 3, we provide some relevant lemmas and prove the positivity property \eqref{pos-prop} and the A-stability property
 \eqref{A} for the BDF$k$ $(k\geqslant3)$ method that are needed for the subsequent stability analysis. In Section \ref{Se:stab}, we use the  multipliers in combination with
the Grenander--Szeg\"o theorem to establish stability for the BDF$k$ corrected scheme by the energy technique.

\section{Multipliers and up to six-step BDF method}\label{sec:1}
Let $N\in \N,$ $\tau:=T/N$ be the time step, and $t_n :=n \tau,
n=0,\dotsc ,N,$ be a uniform partition of the interval $[0,T].$
The  fractional substantial derivative $\partial_{t}^{\alpha,\sigma}\varphi(t_n)$ can be approximated by \cite{Chendeng:13}
\begin{equation}\label{2.1}
 \bar{\partial}_{\tau}^{\alpha,\sigma}\varphi^n:=\frac{1}{{\tau}^{\alpha}}\sum_{j=0}^{n}g_j^k\varphi^{n-j}
\end{equation}
with $\varphi^{n}=\varphi(t_n)$, where the the coefficients $\{g_j^k\}_{j=0}^{\infty}$ are determined by the ($k$-step BDF method) generating power series  $g(\zeta)$,
\begin{equation}\label{2.01}
g(\zeta)=\left(\sum_{j=1}^{k}\frac{1}{j}(1-e^{-\sigma\tau}\zeta)^{j}\right)^\alpha=\sum_{j=0}^{\infty}g_j^k{\zeta}^{j},\quad g_j^k=e^{-\sigma j \tau}l_j^k,
\end{equation}
It should be noted that there are several ways to compute the coefficients $l_j^k$. For example, it can be calculated efficiently by the fast Fourier transform
or recursion in \cite[Chapter 7]{Podlubny:99}; or direct calculation in \cite{CD:13,CD:14}. Here we introduce the simple and efficient formulas to compute $l_j^k$ with linearly computational count, see  Appendix.

We recursively define a sequence of approximations $u^n$ to the nodal values $ u(t_n)$ by the $k$-step BDF method.
Correspondingly, the BDF$k$ scheme for solving \eqref{1.3} seeks approximations $u^n, n=1,...,N$ to the analytic  solution $u(t_n)$  by  \cite{Chendeng:13}
\begin{equation*}
{\bf BDFk~ scheme}~~~P\left(\bar{\partial}_{\tau}\right)(u^n-e^{-\sigma n \tau}\varrho)+Au^n=0,\quad u^0=\varrho.
\end{equation*}
The low regularity of the solution of (1.1) implies the above  standard BDF$k$ scheme  only yields a first-order accuracy \cite{Sakamoto:11,Thomee:06}. To restore the $k$th-order accuracy for BDF$k$,
the BDF$k$ scheme has been  corrected  at the starting $k-1$ steps by \cite{SC:20}
\begin{equation}\label{2.0083}
\begin{split}
{\bf BDFk~corrected~ scheme}~~~~~~~&P\left(\bar{\partial}_{\tau}\right)(u^n-e^{-\sigma n \tau}\varrho)+Au^n=-e^{-\sigma n \tau}a_n^{(k)}A\varrho,  1\leqslant n\leqslant k-1,\\
&P\left(\bar{\partial}_{\tau}\right)(u^n-e^{-\sigma n \tau}\varrho)+Au^n=0,  ~~~~\quad\quad\quad\quad\quad k\leqslant  n\leqslant N,
\end{split}
\end{equation}
where the coefficients $a_n^{(k)}$ are given in Table \ref{table:1}.
Taking $w^n:=u^n-e^{-\sigma n \tau}\varrho$ with $w^0=0$, we can rewrite  \eqref{2.0083} as
\begin{equation}\label{2.0084}
\begin{split}
&P\left(\bar{\partial}_{\tau}\right)w^n+Aw^n=-e^{-\sigma n \tau}\left(1+a_n^{(k)}\right)A\varrho,~~~\quad 1\leqslant n\leqslant k-1,\\
&P\left(\bar{\partial}_{\tau}\right)w^n+Aw^n=-e^{-\sigma n \tau}A \varrho,  \quad\quad\quad\quad\quad\quad  k\leqslant n\leqslant N.
\end{split}
\end{equation}
\begin{table}[h]\fontsize{9.5pt}{15pt}\selectfont
 \begin{center}
  \caption {The coefficients $a_n^{(k)}$.} \vspace{5pt}
\begin{tabular*}{\linewidth}{@{\extracolsep{\fill}}*{1}{|c c c c c c|}}         \hline  
BDF$k$            &$a_1^{(k)}$          &$a_2^{(k)}$          &$a_3^{(k)}$       &$a_4^{(k)}$          & $a_5^{(k)}$   \\ \hline
$k=2$             &$\frac{1}{2}$        &     0               &  0                &  0                 &  0              \\
$k=3$             &$\frac{11}{12}$      &$-\frac{5}{12}$      &  0                &  0                 &  0               \\
$k=4$             &$\frac{31}{24}$      &$-\frac{7}{6}$       &$\frac{3}{8}$      &  0                 &  0                \\
$k=5$             &$\frac{1181}{720}$   &-$\frac{177}{80}$    &$\frac{341}{240}$  &-$\frac{251}{720}$  &  0                 \\
$k=6$             &$\frac{2837}{1440}$  &-$\frac{2543}{720}$  &$\frac{17}{5}$     &-$\frac{1201}{720}$ & $\frac{95}{288}$   \\ \hline
    \end{tabular*}\label{table:1}
  \end{center}
\end{table}
\subsection{Multipliers for BDF methods}

From  the A-stable definition  \cite{HW:10,NO:81}, we introduce the following definition.
\begin{definition}[A-stability]\label{De:As}
Let $g(\zeta)$ be the generating power series of the $k$-step BDF method defined in \eqref{2.01}. Let
$\mu(\zeta)=1-\mu_1e^{-\sigma \tau}\zeta-\dotsb-\mu_k\left(e^{-\sigma \tau}\zeta\right)^k$ be a polynomial,
with real coefficients and roots outside the unit disk.  In addition, the generating power series $g$ and  polynomials $\mu$ have no common divisor. Then, we call the $k$-step scheme described by the pair $(g,\mu)$ is A-stable if
\begin{equation}
\label{A}
\Real \frac {g(\zeta)}{\mu(\zeta)}>0\quad\text{for }\, |\zeta|<1.
\tag{A}
\end{equation}
\end{definition}

Now, $g(\zeta)/\mu(\zeta)$ is holomorphic inside the unit disk in the complex plane, and
$$\lim_{|\zeta|\rightarrow 0}\frac {g(\zeta)}{\mu(\zeta)}=g_0>0.$$
Therefore, using  the maximum principle for harmonic functions, the A-stability property \eqref{A} is equivalent to
\begin{equation*}
\Real \frac {g(\zeta)}{\mu(\zeta)}\geqslant 0 \quad\forall \zeta\in \mathcal{K},
\end{equation*}
with $\mathcal{K}$ the unit circle in the complex plane, $\mathcal{K}:= \{\zeta \in \mathbb{C}:|\zeta|=1\}$.
\begin{definition}[Multipliers]\label{De:mult}
Let $g(\zeta)$ be the generating power series of the $k$-step BDF method defined in \eqref{2.01}.
Consider a $k$-tuple $(\mu_1,\dotsc,\mu_k)$ of real numbers such that
with the given $g(\zeta)$ and
$\mu(\zeta):=1-\mu_1e^{-\sigma \tau}\zeta-\dotsb-\mu_k\left(e^{-\sigma \tau}\zeta\right)^k$, and
the pair $(g,\mu)$ satisfies the A-stability condition \eqref{A}, and, in addition, the generating power series $g$ and  polynomials $\mu$ have no common divisor.
Then, we call $(\mu_1,\dotsc,\mu_k)$ simply \emph{multiplier} if it satisfies the \emph{positivity} property
\begin{equation}
\label{pos-prop}
1-\mu_1e^{-\sigma \tau}\cos x-\dotsb-\mu_ke^{-\sigma k \tau}\cos (kx) >0 \quad \forall x \in \Re. \tag{P}
\end{equation}
\end{definition}

In this paper, the simply multiplier $(\mu_1,\dotsc,\mu_k)$ for the $k$-step BDF method list in the following, see Table \ref{table:2}.
\begin{table}[h]\fontsize{9.5pt}{15pt}\selectfont
 \begin{center}
  \caption {Multipliers for the up to six-step BDF method.} \vspace{5pt}
\begin{tabular*}{\linewidth}{@{\extracolsep{\fill}}*{1}{|c c c c c c c|}}         \hline  
BDF$k$            &$\mu_1$          &$\mu_2$          &$\mu_3$       &$\mu_4$          & $\mu_5$  & $\mu_6$  \\ \hline
$k=3$             &$\frac{1}{2}$        &     0               &  0                &                   &    &             \\
$k=4$             &$\frac{1}{2}$      &0      &0      &  0         &        &                 \\
$k=5$             &$1$   &-$\frac{1}{4}$    &0  &0  &  0         &        \\
$k=6$             &$\frac{43}{30}$  &-$\frac{2}{3}$  &$\frac{1}{10}$     &0 &0   &0 \\ \hline
    \end{tabular*}\label{table:2}
  \end{center}
\end{table}

\begin{remark}\label{rem:NO-six}
From \cite{HW:10,NO:81}, we know that the  A-stability is equivalent to G-stability for parabolic equation.
However, this equivalence property is still an open question for subdiffusion equation
and awaits further investigation.
\end{remark}

For simplicity, we denote by $\langle\cdot,\cdot\rangle$ the inner product on $V,$ $\langle w, v\rangle:=(A^{1/2}w, A^{1/2}v).$
To prove stability of the method by the energy technique, we test \eqref{2.0084}
by $v^n=w^{n}-\mu_1e^{-\sigma \tau}w^{n-1}-\dotsb-\mu_ke^{-\sigma k\tau}w^{n-k}$ and obtain
\begin{equation}\label{2.4}
\begin{split}
\left(P\left(\bar{\partial}_{\tau}\right)w^n,v^n\right)+\left\langle w^n,v^n\right\rangle=-e^{-\sigma n \tau}\left(1+a_n^{(k)}\right)\left\langle\varrho,v^n\right\rangle.
\end{split}
\end{equation}
For the first term on the left hand side of \eqref{2.4}, from \eqref{1.4}, we have
\begin{equation}\label{2.a4}
\begin{split}
\left(P\left(\bar{\partial}_{\tau}\right)w^n,v^n\right)&=\left(\int_0^1\bar{\partial}_{\tau}^{\alpha,\sigma} w^n d\nu\left(\alpha\right),v^n\right)=\int_0^1\left(\bar{\partial}_{\tau}^{\alpha,\sigma} w^n ,v^n\right)d\nu\left(\alpha\right)\\
&=\frac{1}{\tau^{\alpha}}\int_0^1\left(\sum_{j=0}^n g_j^k w^{n-j},v^n\right)d\nu\left(\alpha\right).
\end{split}
\end{equation}
We consider the integrand of \eqref{2.a4} for our subsequent discussion.

Case 1: $k=3,4$. Taking $v^n=w^n-\mu_1e^{-\sigma \tau}w^{n-1}$ with $\mu_1=1/2$, we have
\begin{equation*}
\begin{split}
\sum_{j=0}^n g_j^k w^{n-j}
&=g_0^k\left(w^{n}-\mu_1e^{-\sigma \tau}w^{n-1}\right)+\left(g_1^k+\frac{1}{2}e^{-\sigma \tau}g_0^k\right)\left(w^{n-1}-\mu_1e^{-\sigma \tau}w^{n-2}\right)+\cdots\\
&\quad+\left(g_{n-1}^k+\frac{1}{2}e^{-\sigma \tau}g_{n-2}^k+\frac{1}{2^2}e^{-2\sigma \tau} g_{n-3}^k+\cdots+\frac{1}{2^{n-1}}e^{-\left(n-1\right)\sigma \tau}g_{0}^k\right)\times\\
&\qquad\left(w^{1}-\mu_1e^{-\sigma \tau}w^{0}\right)=\sum_{j=0}^{n-1} q_j^k v^{n-j}
\end{split}
\end{equation*}
with
\begin{equation}\label{2.6}
\begin{split}
q_j^k=\sum_{m=0}^j\frac{1}{2^{m}}e^{-\sigma m \tau}g_{j-m}^k.
\end{split}
\end{equation}

Case 2: $k=5$. Taking $v^n=w^n-\mu_1e^{-\sigma \tau}w^{n-1}-\mu_2e^{-2\sigma \tau}w^{n-2}$ with $\mu_1=1$, $\mu_2=-1/4$, it yields
\begin{equation*}
\begin{split}
\sum_{j=0}^n g_j^k w^{n-j}
&=g_0^k\left(w^n-\mu_1e^{-\sigma \tau}w^{n-1}-\mu_2e^{-2\sigma \tau}w^{n-2}\right)+\left(g_1^k+e^{-\sigma \tau}g_0^k\right)\times\\
&\qquad\left(w^{n-1}-\mu_1e^{-\sigma \tau}w^{n-2}-\mu_2e^{-2\sigma \tau}w^{n-3}\right)+\cdots+\\
&\qquad\left(g_{n-1}^k+e^{-\sigma \tau}g_{n-2}^k+\frac{3}{2^2}e^{-2\sigma \tau}g_{n-3}^k+\cdots+\frac{n}{2^{n-1}}e^{-\left(n-1\right)\sigma \tau}g_{0}^k\right)\times\\
&\qquad\left(w^{1}-\mu_1e^{-\sigma \tau}w^{0}-\mu_2e^{-2\sigma \tau}w^{-1}\right)=\sum_{j=0}^{n-1} q_j^k v^{n-j}
\end{split}
\end{equation*}
with the  starting values $w^{-1}=0$ \cite{Lu:88}   and
\begin{equation}\label{2.7}
\begin{split}
q_j^k=\sum_{m=0}^j\frac{m+1}{2^{m}}e^{-\sigma m \tau}g_{j-m}^k.
\end{split}
\end{equation}

Case 3: $k=6$. Taking $v^n=w^n-\mu_1e^{-\sigma \tau}w^{n-1}-\mu_2e^{-2\sigma \tau}w^{n-2}-\mu_3e^{-3\sigma \tau}w^{n-3}$ with $\mu_1=43/30$, $\mu_2=-2/3$, $\mu_3=1/10$, there exists
\begin{equation*}
\begin{split}
&\sum_{j=0}^n g_j^k w^{n-j}\\
&=g_0^k\left(w^n-\mu_1e^{-\sigma \tau}w^{n-1}-\mu_2e^{-2\sigma \tau}w^{n-2}-\mu_3e^{-3\sigma \tau}w^{n-3}\right)\\
&\!+\!\left(g_1^k+\frac{43}{30}e^{-\sigma \tau}g_0^k\right)\left(w^{n-1}-\mu_1e^{-\sigma \tau}w^{n-2}-\mu_2e^{-2\sigma \tau}w^{n-3}-\mu_3e^{-3\sigma \tau}w^{n-4}\right)\!+\!\cdots\\
&+\left(g_{n-1}^k+\frac{43}{30}e^{-\sigma \tau}g_{n-2}^k+\!\cdots\!
+\frac{243\times 18^{n-2}-15^{n}+25\times10^{n-2}}{30^{n-1}}e^{-\left(n-1\right)\sigma \tau}g_{0}^k\right)\\
&\quad\times\left(w^{1}-\mu_1e^{-\sigma \tau}w^{0}-\mu_2e^{-2\sigma \tau}w^{-1}-\mu_3e^{-3\sigma \tau}w^{-2}\right)
=\sum_{j=0}^{n-1} q_j^k v^{n-j}
\end{split}
\end{equation*}
with the  starting values $w^{-1}=w^{-2}=0$ and
\begin{equation}\label{2.8}
\begin{split}
q_j^k=\sum_{m=0}^j\frac{243\times 18^{m-1}-15^{m+1}+25\times10^{m-1}}{30^{m}}e^{-\sigma m \tau}g_{j-m}^k.
\end{split}
\end{equation}
Therefore, E.q. (\ref{2.a4}) can be written in the following equivalent form
\begin{equation}\label{2.9}
\begin{split}
\left(P\left(\bar{\partial}_{\tau}\right)w^n,v^n\right)=\frac{1}{\tau^{\alpha}}\int_0^1\left(\sum_{j=0}^{n-1} q_j^k v^{n-j},v^n\right)d\nu\left(\alpha\right).
\end{split}
\end{equation}

To prove stability of the method by the energy technique for (\ref{2.4}).
 First, we sum over $n$ and subsequently estimate the sum of each
terms. The integrand of the first term on the left-hand side can be estimated
from below using Lemmas \ref{lemma4.1}-\ref{lemma4.003}; this is the motivation for the requirement \eqref{A}.
In view of the Grenander--Szeg\"o theorem, the condition \eqref{pos-prop} ensures that symmetric band Toeplitz matrices, with generating function the positive trigonometric polynomial
$1-\mu_1e^{-\sigma \tau}\cos x-\dotsb-\mu_ke^{-\sigma k \tau}\cos (kx)$, are positive definite;
see Lemma \ref{lemma3.1}.

\section{Positivity property (P) and A-stability property (A) }
Before we proceed, for the reader's convenience, we recall the notion of the generating function of
an $n\times n$ Toeplitz  matrix $T_n$ as well as an auxiliary result, the Grenander--Szeg\"o theorem.

\begin{definition}\cite[p.\,27]{Quarteroni:07}\label{definition2.7}
A matrix $A \in \Re^{n\times n}$ is said to be positive definite in $\Re^{n}$ if $(Ax,x)>0$, $\forall x \in \Re^{n}$, $x\neq 0$.
\end{definition}

\begin{lemma}\cite[p.\,28]{Quarteroni:07}\label{lemma2.4}
A real matrix $A$ of order $n$ is positive definite  if and only if  its symmetric part $H=\frac{A+A^T}{2}$ is positive definite.
Let $H \in \Re^{n\times n}$ be symmetric. Then $H$ is positive definite if and only if the eigenvalues of $H$ are positive.
\end{lemma}

\begin{definition}\cite[p.\,13]{Chan:07} (the generating function of a Toeplitz matrix) \label{De:gen-funct}
Consider the $n \times n$ Toeplitz  matrix  $T_n=(t_{ij})\in \Re^{n,n}$ with diagonal entries $t_0,$ subdiagonal entries
$t_1,$ superdiagonal entries $t_{-1},$ and so on, and $(n,1)$ and $(1,n)$ entries
$t_{n-1}$ and   $t_{1-n}$, respectively, i.e., the entries $t_{ij}=t_{i-j}, i,j=1,\dotsc,n,$ are constant along the diagonals of $T_n.$
 Let   $t_{-n+1},\dotsc, t_{n-1}$ be the Fourier coefficients of the trigonometric polynomial $f(x)$, i.e.,
\begin{equation*}
  t_k=\frac{1}{2\pi}\int_{-\pi}^{\pi}f(x)e^{-i kx} \mathrm{d} x,\quad k=1-n,\dotsc,n-1.
\end{equation*}
Then, $f(x)=\sum_{k=1-n}^{n-1} t_ke^{i kx},$ is called  \emph{generating function} of $T_n$.
\end{definition}

\begin{lemma}\cite[p.\,13--15]{Chan:07} (the Grenander-Szeg\"{o} theorem)\label{lemma2.6}
Let $T_n$ be given in Definition \ref{De:gen-funct} with a generating function $f(x)$.
Then, the smallest and largest eigenvalues  $\lambda_{\min}(T_n)$ and $\lambda_{\max}(T_n)$, respectively, of $T_n$ are bounded as follows
\begin{equation*}
  f_{\min} \leqslant \lambda_{\min}(T_n) \leqslant \lambda_{\max}(T_n) \leqslant f_{\max},
\end{equation*}
with $f_{\min}$ and  $f_{\max}$  the minimum and maximum of $f(x)$, respectively.
In particular, if  $f_{\min}$ is positive, then $T_n$ is positive definite.
\end{lemma}

\subsection{Checking the positivity property (P) in space direction}
In this subsection, we   prove the positivity property \eqref{pos-prop} in space direction.
\begin{lemma}\label{lemma3.1}
For any positive integer $N$, it holds that
\begin{equation*}
\sum_{n=1}^{N}\left\langle w^n,w^n-\sum_{j=1}^k \mu_j e^{-\sigma j \tau}w^{n-j}\right\rangle\geqslant c_k\sum_{n=1}^{N}\|w^n\|^2,~k=3,4,5,6,
\end{equation*}
where $\left( c_3,c_4,c_5,c_6 \right)=\left(\frac{1}{2},\frac{1}{2},\frac{1}{4},\frac{1}{24}  \right)$.
\end{lemma}
\begin{proof}
We prove the desired results by the following three cases.

Case $1$: Let $k=3,4$. With this notation $\mu_0:=-1/2$, $\mu_1:=1/2$,
$\mu_2=\mu_3=\mu_4=0$, we have
\begin{equation*}
\begin{split}
\sum_{n=1}^{N}\left\langle w^n,w^n-\mu_1e^{-\sigma \tau} w^{n-1}\right\rangle=\frac{1}{2}\sum_{n=1}^{N}\|w^n\|^2+\sum_{i,j=1}^N \ell_{i,j}\left\langle w^i,w^j\right\rangle.
\end{split}
\end{equation*}
Here the lower triangular Toeplitz matrix $L=(\ell_{ij})\in \Re^{N,N}$  with entries
\[\ell_{i,i-j}=-\mu_je^{-\sigma j \tau}, \quad j=0,1, \quad i=j+1,\dotsc,N,\]
and all other entries equal zero.  From Definition \ref{De:gen-funct}, we know that the generating function of
$(L+L^T)/2$ is
\begin{equation*}
\begin{split}
f(x)=\frac{1}{2}\left(1-e^{-\sigma \tau}\cos x\right) \quad \forall x \in \Re.
\end{split}
\end{equation*}
Using Lemma \ref{lemma2.4} and  \ref{lemma2.6}, it implies that $L$ is semipositive definite.  Then we have
\begin{equation*}
\begin{split}
\sum_{n=1}^{N}\left\langle w^n,w^n-\mu_1 e^{-\sigma \tau}w^{n-1}\right\rangle\geqslant\frac{1}{2}\sum_{n=1}^{N}\|w^n\|^2.
\end{split}
\end{equation*}

Case $2$: Let $k=5$. With this notation $\mu_0:=-3/4$, $\mu_1:=1$, $\mu_2:=-1/4$, $\mu_3=\mu_4=\mu_5=0$, we obtain
\begin{equation*}
\begin{split}
\sum_{n=1}^{N}\left\langle w^n,w^n-\mu_1 e^{-\sigma \tau}w^{n-1}-\mu_2e^{-2\sigma \tau}w^{n-2}\right\rangle=\frac{1}{4}\sum_{n=1}^{N}\|w^n\|^2+\sum_{i,j=1}^N \ell_{i,j}\left\langle w^i,w^j\right\rangle.
\end{split}
\end{equation*}
Here, we introduce the lower triangular Toeplitz matrix $L_1=(\ell_{ij})\in \Re^{N,N}$  with entries
\[\ell_{i,i-j}=-\mu_je^{-\sigma j \tau}, \quad j=0,1,2, \quad i=j+1,\dotsc,N,\]
and all other entries equal zero.
From Definition \ref{De:gen-funct}, we know that the generating function  of $(L_1+L_1^T)/2$ is
\begin{equation*}
\begin{split}
f(x)=\frac{3}{4}-e^{-\sigma \tau}\cos x+\frac{1}{4}e^{-2\sigma \tau}\cos(2x)\geqslant\frac{1}{2}\left(1-e^{-\sigma \tau}\cos x\right)^2 \quad \forall x \in \Re.
\end{split}
\end{equation*}
Using Lemma \ref{lemma2.4} and  \ref{lemma2.6}, it implies that $L_1$ is semipositive definite.  Then we have
\begin{equation*}
\begin{split}
\sum_{n=1}^{N}\left\langle w^n,w^n-\mu_1e^{-\sigma \tau} w^{n-1}-\mu_2e^{-2\sigma \tau}w^{n-2}\right\rangle\geqslant\frac{1}{4}\sum_{n=1}^{N}\|w^n\|^2.
\end{split}
\end{equation*}

Case $3$: Let $k=6$. With this notation $\mu_0:=-23/24$, $\mu_1:=43/30$, $\mu_2:=-2/3$, $\mu_3:=1/10$, $\mu_4=\mu_5=\mu_6=0$, it yields
\begin{equation*}
\begin{split}
&\sum_{n=1}^{N}\left\langle w^n,w^n-\mu_1e^{-\sigma  \tau} w^{n-1}-\mu_2e^{-2\sigma \tau}w^{n-2}-\mu_3e^{-3\sigma \tau}w^{n-3}\right\rangle\\
&=\frac{1}{24}\sum_{n=1}^{N}\|w^n\|^2+\sum_{i,j=1}^N \ell_{i,j}\left\langle w^i,w^j\right\rangle.
\end{split}
\end{equation*}
To this end, we introduce the lower triangular Toeplitz matrix $L_2=(\ell_{ij})\in \Re^{N,N}$  with entries
\[\ell_{i,i-j}=-\mu_je^{-\sigma j\tau}, \quad j=0,1,2,3, \quad i=j+1,\dotsc,N,\]
and all other entries equal zero.
According to Definition \ref{De:gen-funct},  the generating function of $(L_2+L_2^T)/2$ is
\begin{equation*}
\begin{split}
f(x)&=\frac{23}{24} -\frac{43}{30}e^{-\sigma \tau}\cos x+\frac{2}{3}e^{-2\sigma\tau}\cos(2x)-\frac{1}{10}e^{-3\sigma \tau}\cos(3x)\\
&=-\frac{2}{5}\left(e^{-\sigma \tau}\cos x\right)^3+\frac{4}{3}\left(e^{-\sigma \tau}\cos x\right)^2-\frac{43}{30}e^{-\sigma \tau}\cos x+\frac{3}{10}e^{-3\sigma \tau}\cos x\\
&\quad-\frac{2}{3}e^{-2\sigma \tau}+\frac{23}{24}\quad \forall x \in \Re.
\end{split}
\end{equation*}
Let $\xi=e^{-\sigma \tau}\cos x$, $\lambda=e^{-2\sigma \tau}$, it leads to
\begin{equation*}
\begin{split}
f(x)=-\frac{2}{5}\xi^3+\frac{4}{3}\xi^2-\frac{43}{30}\xi+\frac{3}{10}\lambda\xi-\frac{2}{3}\lambda+\frac{23}{24},\quad \xi\in [-1,1], \quad \lambda\in (0,1].
\end{split}
\end{equation*}
Clearly, $f(x)$ is decreasing with respect to $\lambda$, it yields
\begin{equation*}
\begin{split}
f(x)\geqslant-\frac{2}{5}\xi^3+\frac{4}{3}\xi^2-\frac{17}{15}\xi+\frac{7}{24},\quad \xi\in [-1,1].
\end{split}
\end{equation*}
Hence, we consider the polynomial $p,$
\begin{equation*}
p(\xi):=-\frac{2}{5}\xi^3+\frac{4}{3}\xi^2-\frac{17}{15}\xi+\frac{7}{24},\quad \xi\in [-1,1].
\end{equation*}
It is easily seen that $p$ attains its minimum at $\xi^\star=(20-\sqrt{94})/18$ 
and
\[p(\xi^\star)>0.004785 > 0.\]
%
Using Lemma \ref{lemma2.4} and  \ref{lemma2.6}, it implies that $L_2$ is positive definite.  Then we obtain
\begin{equation*}
\begin{split}
\sum_{n=1}^{N}\left\langle w^n,w^n-\mu_1e^{-\sigma \tau} w^{n-1}-\mu_2e^{-2\sigma \tau}w^{n-2}-\mu_3e^{-3\sigma \tau}w^{n-3}\right\rangle\geqslant\frac{1}{24}\sum_{n=1}^{N}\|w^n\|^2.
\end{split}
\end{equation*}
The proof is completed.
\end{proof}

\subsection{Checking the A-stability property (A) in time direction}
The G-stability theory or the  equivalence of A-stability and G-stability are used to analyze the stability of BDF methods for parabolic equation.
However, it is not easy to extend to subdiffusion equation.
In this subsection, we prove the A-stability property \eqref{A} in time direction for subdiffusion equation.
It is well-known that  BDF$k$ ($k\geqslant 3$) is not  $A$-stable \cite[p.251]{HW:10}, which underlies the main technical difficulty in carrying out a rigorous stability  analysis  by the energy technique.
Based on the generating power series $g(\zeta)$ in \eqref{2.01}, we construct a novel generating power series
\begin{equation}\label{3.001}
q(\zeta)=g(\zeta)\Big/\mu(\zeta)=\sum_{j=0}^\infty q_j^k\zeta^j,
\end{equation}
where $q_j^k$ is given in \eqref{2.6}-\eqref{2.8}.
According to \cite{CD:18,Ji:15,Xu:11} and Definition \ref{De:As}, we have the following results.
\begin{lemma}\label{lemma2.7}
Let $\{ q_j^k\}_{j=0}^\infty$ be a sequence of real numbers such that $q(\zeta)=\sum_{j=0}^\infty q_j^k\zeta^j$ is analytic in the unit disk $S=\{\zeta \in \mathbb{C}: |\zeta|\leqslant 1\}$.
Then for any positive integer $N$ and for any $\left(v^1,\ldots,v^N\right)$
\begin{equation*}
\begin{split}
\sum_{n=1}^N\left(\sum_{j=0}^{n-1} q_j^k v^{n-j},v^n\right)\geqslant 0,
\end{split}
\end{equation*}
if and only if
$\Real q(\zeta)\geqslant 0$, if and only if $\arg \left(q(\zeta)\right)\in\left[-\frac{\pi}{2},\frac{\pi}{2}\right]$.
\end{lemma}

Corresponding, the $k$-step schemes described by the pair $(g,\mu)$ are all  $A$-stable, see the following  Lemmas \ref{lemma4.1}-\ref{lemma4.003}.

\begin{lemma} (BDF3)\label{lemma4.1}
Let $q_j^k$ be defined by (\ref{2.6}). Then for any positive integer $N$, it holds that
\begin{equation*}
\begin{split}
\sum_{n=1}^N\left(\sum_{j=0}^{n-1} q_j^k v^{n-j},v^n\right)\geqslant 0.
\end{split}
\end{equation*}
\end{lemma}
\begin{proof}
From \eqref{2.6} and \eqref{3.001}, we can check that
\begin{equation*}
\begin{split}
q(\zeta)=\frac{\left(\frac{11}{6}-3e^{-\sigma\tau}\zeta+\frac{3}{2}(e^{-\sigma\tau}\zeta)^2-\frac{1}{3}(e^{-\sigma\tau}\zeta)^3\right)^{\alpha}}{1-\frac{1}{2}e^{-\sigma\tau}\zeta}.
\end{split}
\end{equation*}
Taking $z=e^{-\sigma\tau}\zeta$, it leads to
\begin{equation*}
\begin{split}
\frac{\left(\frac{11}{6}-3z+\frac{3}{2}z^2-\frac{1}{3}z^3\right)^{\alpha}}{1-\frac{1}{2}z}=
\frac{\left(1-z\right)^{\alpha}\left(\frac{11}{6}-\frac{7}{6}z+\frac{2}{6}z^2\right)^{\alpha}}{1-\frac{1}{2}z}.
\end{split}
\end{equation*}

Next we apply  the  Grenander-Szeg\"{o} theorem  to obtain the desired result.
Let  $z=e^{ix}$. Here  we just need to consider its principal value on $x\in [0,\pi]$, since $x\in [\pi,2\pi]$ can be similarly discussed.
Then
\begin{equation*}
\begin{split}
\left(1-z\right)^{\alpha}=\left(2\sin\frac{x}{2}\right)^{\alpha}e^{i{\alpha}\theta_1}
\end{split}
\end{equation*}
with $\theta_1=\arctan\frac{-\sin (x)}{1-\cos x}=\frac{x-\pi}{2}\leqslant 0$.
\begin{equation*}
\begin{split}
\left(\frac{11}{6}-\frac{7}{6}z+\frac{2}{6}z^2\right)^{\alpha}=\left(a_3-ib_3\right)^{\alpha}=\left(a_3^2+b_3^2\right)^{\frac{\alpha}{2}}e^{i{\alpha}\theta_2},
\end{split}
\end{equation*}
where
\begin{equation*}
\begin{split}
a_3&=\frac{1}{6}\left(11-7\cos x+2\cos(2x)\right)>0,~ b_3=\frac{1}{6}\left(7\sin x-2\sin(2x)\right)\geqslant 0, \\ \theta_2&=\arctan\frac{-\left(7\sin x-2\sin(2x)\right)}{11-7\cos x+2\cos(2x)}\leqslant 0.
\end{split}
\end{equation*}
Moreover, we get
\begin{equation*}
\begin{split}
\frac{1}{1-\frac{1}{2}z}=\frac{1}{\sqrt{\frac{5}{4}-\cos x}}e^{i\theta_3},~~\theta_3=\arctan\frac{\frac{1}{2}\sin (x)}{1-\frac{1}{2}\cos x}\leqslant \arctan1= \frac{\pi}{4}.
\end{split}
\end{equation*}
From Lemma \ref{lemma2.7}, we need to prove
\begin{equation*}
\begin{split}
\Real\left\{\frac{\left(\frac{11}{6}-3z+\frac{3}{2}z^2-\frac{1}{3}z^3\right)^{\alpha}}{1-\frac{1}{2}z}\right\}\geqslant0,
\end{split}
\end{equation*}
which is equal to prove
\begin{equation*}
\begin{split}
\arg\left\{\frac{\left(\frac{11}{6}-3z+\frac{3}{2}z^2-\frac{1}{3}z^3\right)^{\alpha}}{1-\frac{1}{2}z}\right\}\in\left[-\frac{\pi}{2},\frac{\pi}{2}\right].
\end{split}
\end{equation*}
According to the above equations, we have
\begin{equation*}
\begin{split}
&\arg\left\{\frac{\left(\frac{11}{6}-3z+\frac{3}{2}z^2-\frac{1}{3}z^3\right)^{\alpha}}{1-\frac{1}{2}z}\right\}\\
&=\arg\left\{\left(1-z\right)^{\alpha}\right\}+\arg\left\{\left(\frac{11}{6}-\frac{7}{6}z+\frac{2}{6}z^2\right)^{\alpha}\right\}
+\arg\left\{\frac{1}{1-\frac{1}{2}z}\right\}\\&=\alpha \theta_1+\alpha \theta_2+\theta_3.
\end{split}
\end{equation*}
Since $\alpha \theta_1+\alpha \theta_2+\theta_3\leqslant \theta_3\leqslant  \frac{\pi}{4}$ and $\alpha \theta_1+\alpha \theta_2+\theta_3\geqslant \theta_1+\theta_2+\theta_3$.
Next we just need to prove $g(x)=\left(\theta_1+\theta_2+\theta_3\right)(x)\geqslant -\frac{\pi}{2}$.
With $y:=\cos x,$ it yields
\begin{equation*}
\begin{split}
g'(x)=\frac{4}{\left(88y^2-182y+130\right)\left(5-4y\right)}h(y)
\end{split}
\end{equation*}
with $h(y)=-88y^3+262y^2-230y+65,~y\in [-1,1]$.
It is easily seen that $h$ attains its minimum at $y^\star=(131-\sqrt{1981})/132$ 
and
\[h(y^\star)> 2.02 > 0.\]
Moreover, combining with $\left(88y^2-182y+130\right)\left(5-4y\right)>0$, it implies that $g'(x)$ is positive, $g(x)\geqslant g(0)=-\frac{\pi}{2}$.
The proof is completed.
\end{proof}

\begin{lemma}(BDF4) \label{lemma4.001}
Let $q_j^k$ be defined by (\ref{2.6}). Then for any positive integer $N$, it holds that
\begin{equation*}
\begin{split}
\sum_{n=1}^N\left(\sum_{j=0}^{n-1} q_j^k v^{n-j},v^n\right)\geqslant 0.
\end{split}
\end{equation*}
\end{lemma}
\begin{proof}
From  \eqref{2.6} and \eqref{3.001}, there exists
\begin{equation*}
\begin{split}
q(\zeta)=\frac{\left(\frac{25}{12}-4e^{-\sigma\tau}\zeta+3(e^{-\sigma\tau}\zeta)^2-\frac{4}{3}(e^{-\sigma\tau}\zeta)^3+\frac{1}{4}(e^{-\sigma\tau}\zeta)^4\right)^{\alpha}}{1-\frac{1}{2}e^{-\sigma\tau}\zeta}.
\end{split}
\end{equation*}
Taking $z=e^{-\sigma\tau}\zeta$, it leads to
\begin{equation*}
\begin{split}
\frac{\left(\frac{25}{12}-4z+3z^2-\frac{4}{3}z^3+\frac{1}{4}z^4\right)^{\alpha}}{1-\frac{1}{2}z}=
\frac{\left(1-z\right)^{\alpha}\left(\frac{25}{12}-\frac{23}{12}z+\frac{13}{12}z^2-\frac{1}{4}z^3\right)^{\alpha}}{1-\frac{1}{2}z}.
\end{split}
\end{equation*}

Next we apply  the  Grenander-Szeg\"{o} theorem  to obtain the desired result.
Let  $z=e^{ix}$. Here  we just need to consider its principal value on $x\in [0,\pi]$.
It is easy to compute that
\begin{equation*}
\begin{split}
\left(1-z\right)^{\alpha}=\left(2\sin\frac{x}{2}\right)^{\alpha}e^{i{\alpha}\theta_1}
\end{split}
\end{equation*}
with $\theta_1=\arctan\frac{-\sin (x)}{1-\cos x}=\frac{x-\pi}{2}\leqslant 0$; and
\begin{equation*}
\begin{split}
\left(\frac{25}{12}-\frac{23}{12}z+\frac{13}{12}z^2-\frac{1}{4}z^3\right)^{\alpha}=\left(a_4-ib_4\right)^{\alpha}=\left(a_4^2+b_4^2\right)^{\frac{\alpha}{2}}e^{i{\alpha}\theta_2}
\end{split}
\end{equation*}
with
\begin{equation*}
\begin{split}
a_4&=\frac{1}{12}\left(25-23\cos x+13\cos(2x)-3\cos(3x)\right)>0,\\
b_4&=\frac{1}{12}\left(23\sin x-13\sin(2x)+3\sin(3x)\right)\geqslant 0,\\
\theta_2&=\arctan\frac{-\left(23\sin x-13\sin(2x)+3\sin(3x)\right)}{25-23\cos x+13\cos(2x)-3\cos(3x)}\leqslant0.
\end{split}
\end{equation*}
Moreover, we get
\begin{equation*}
\begin{split}
\frac{1}{1-\frac{1}{2}z}=\frac{1}{\sqrt{\frac{5}{4}-\cos x}}e^{i\theta_3},~~\theta_3=\arctan\frac{\frac{1}{2}\sin (x)}{1-\frac{1}{2}\cos x}\leqslant \arctan1= \frac{\pi}{4}.
\end{split}
\end{equation*}
From Lemma \ref{lemma2.7}, we need to prove
\begin{equation*}
\begin{split}
\Real\left\{\frac{\left(\frac{25}{12}-4z+3z^2-\frac{4}{3}z^3+\frac{1}{4}z^4\right)^{\alpha}}{1-\frac{1}{2}z}\right\}\geqslant0,
\end{split}
\end{equation*}
which is equal to prove
\begin{equation*}
\begin{split}
\arg\left\{\frac{\left(\frac{25}{12}-4z+3z^2-\frac{4}{3}z^3+\frac{1}{4}z^4\right)^{\alpha}}{1-\frac{1}{2}z}\right\}\in\left[-\frac{\pi}{2},\frac{\pi}{2}\right].
\end{split}
\end{equation*}
According to the above equations, we get
\begin{equation*}
\begin{split}
&\arg\left\{\frac{\left(\frac{25}{12}-4z+3z^2-\frac{4}{3}z^3+\frac{1}{4}z^4\right)^{\alpha}}{1-\frac{1}{2}z}\right\}\\
&=\arg\left\{\left(1-z\right)^{\alpha}\right\}\!+\arg\left\{\left(\frac{25}{12}-\frac{23}{12}z+\frac{13}{12}z^2-\frac{1}{4}z^3\right)^{\alpha}\right\}
+\arg\left\{\frac{1}{1-\frac{1}{2}z}\right\}\\
&=\alpha \theta_1+\alpha \theta_2+\theta_3.
\end{split}
\end{equation*}
Since $\alpha \theta_1+\alpha \theta_2+\theta_3\leqslant \theta_3\leqslant  \frac{\pi}{4}$ and $\alpha \theta_1+\alpha \theta_2+\theta_3\geqslant \theta_1+\theta_2+\theta_3$.
Next we only need to prove $g(x)=\left(\theta_1+\theta_2+\theta_3\right)(x)\geqslant -\frac{\pi}{2}$.
With $y:=\cos x,$ we have
\begin{equation*}
\begin{split}
g'(x)=\frac{8}{\left(-600y^3+1576y^2-1376y+544\right)\left(5-4y\right)}h(y).
\end{split}
\end{equation*}
Here $h(y)=450y^4-1601y^3+1949y^2-862y+82, y\in(-1,1)$ has the roots
$y_1\approx 0.1288$ with $x_1\approx 1.4416$, $y_2\approx 0.76042$ with $x_2\approx 0.7068$.
In further, we have $h(y)>0$ if $y\in(-1,y_1)$ and $h(y)<0$ if  $y\in(y_1,y_2)$ and $h(y)>0$ if  $y\in(y_2,1)$.
Moreover, combining with $\left(-600y^3+1576y^2-1376y+544\right)\left(5-4y\right)>0$, it implies  that
 $g'(x)>0$ if $x\in(x_1,\pi)$ and $g'(x)<0$ if  $x\in(x_2,x_1)$ and $g'(x)>0$ if  $x\in(0,x_2)$.
Therefore, it is seen that $g$ attains its minimum at $x^\star=x_1$ 
and
\[g(x^\star)>-1.37>-\frac{\pi}{2}.\]
On the other hand, it can be easily checked that $g(0)=-\frac{\pi}{2}$ and $g(\pi)=0$.
Hence, we have $g(x)\geqslant-\frac{\pi}{2}$.
The proof is completed.
\end{proof}

\begin{lemma}(BDF5) \label{lemma4.002}
Let $q_j^k$ be defined by (\ref{2.7}). Then for any positive integer $N$, it holds that
\begin{equation*}
\begin{split}
\sum_{n=1}^N\left(\sum_{j=0}^{n-1} q_j^k v^{n-j},v^n\right)\geqslant 0.
\end{split}
\end{equation*}
\end{lemma}
\begin{proof}
From \eqref{2.7} and \eqref{3.001}, it yields
\begin{equation*}
\begin{split}
q(\zeta)=\frac{\left(\frac{137}{60}-5e^{-\sigma\tau}\zeta+5(e^{-\sigma\tau}\zeta)^2-\frac{10}{3}(e^{-\sigma\tau}\zeta)^3+\frac{5}{4}(e^{-\sigma\tau}\zeta)^4-\frac{1}{5}(e^{-\sigma\tau}\zeta)^5\right)^{\alpha}}{\left(1-\frac{1}{2}e^{-\sigma\tau}\zeta\right)^2}.
\end{split}
\end{equation*}
Taking $z=e^{-\sigma\tau}\zeta$, it leads to
\begin{equation*}
\begin{split}
\frac{\left(\frac{137}{60}-5z+5z^2-\frac{10}{3}z^3+\frac{5}{4}z^4-\frac{1}{5}z^5\right)^{\alpha}}{\left(1-\frac{1}{2}z\right)^2}=
\frac{\left(1-z\right)^{\alpha}\left(\frac{137}{60}-\frac{163}{60}z+\frac{137}{60}z^2-\frac{63}{60}z^3+\frac{1}{5}z^4\right)^{\alpha}}{\left(1-\frac{1}{2}z\right)^2}.
\end{split}
\end{equation*}

We next apply  the  Grenander-Szeg\"{o} theorem  to prove the desired result.
Let  $z=e^{ix}$ with $x\in [0,\pi]$, it yields
\begin{equation*}
\begin{split}
\left(1-z\right)^{\alpha}=\left(2\sin\frac{x}{2}\right)^{\alpha}e^{i{\alpha}\theta_1}
\end{split}
\end{equation*}
with $\theta_1=\arctan\frac{-\sin (x)}{1-\cos x}=\frac{x-\pi}{2}\leqslant 0$; and
\begin{equation*}
\begin{split}
\left(\frac{137}{60}-\frac{163}{60}z+\frac{137}{60}z^2-\frac{63}{60}z^3+\frac{1}{5}z^4\right)^{\alpha}
=\left(a_5-ib_5\right)^{\alpha}=\left(a_5^2+b_5^2\right)^{\frac{\alpha}{2}}e^{i{\alpha}\theta_2}
\end{split}
\end{equation*}
with
\begin{equation*}
\begin{split}
a_5&=\frac{1}{60}\left(137-163\cos x+137\cos(2x)-63\cos(3x)+12\cos(4x)\right)>0,\\
b_5&=\frac{1}{60}\left(163\sin x-137\sin(2x)+63\sin(3x)-12\sin(4x)\right)\geqslant 0,\\
\theta_2&=\arctan\frac{-\left(163\sin x-137\sin(2x)+63\sin(3x)-12\sin(4x)\right)}{137-163\cos x+137\cos(2x)-63\cos(3x)+12\cos(4x)}\leqslant0.
\end{split}
\end{equation*}
Moreover, we get
\begin{equation*}
\begin{split}
\frac{1}{\left(1-\frac{1}{2}z\right)^2}=\frac{1}{\frac{5}{4}-\cos x}e^{i\theta_3},~~\theta_3=2\arctan\frac{\frac{1}{2}\sin (x)}{1-\frac{1}{2}\cos x}\leqslant 2\arctan1= \frac{\pi}{2}.
\end{split}
\end{equation*}
From Lemma \ref{lemma2.7}, we need to prove
\begin{equation*}
\begin{split}
\Real\left\{\frac{\left(\frac{137}{60}-5z+5z^2-\frac{10}{3}z^3+\frac{5}{4}z^4-\frac{1}{5}z^5\right)^{\alpha}}{\left(1-\frac{1}{2}z\right)^2}\right\}\geqslant0,
\end{split}
\end{equation*}
which is equal to prove
\begin{equation*}
\begin{split}
\arg\left\{\frac{\left(\frac{137}{60}-5z+5z^2-\frac{10}{3}z^3+\frac{5}{4}z^4-\frac{1}{5}z^5\right)^{\alpha}}{\left(1-\frac{1}{2}z\right)^2}\right\}\in\left[-\frac{\pi}{2},\frac{\pi}{2}\right].
\end{split}
\end{equation*}
According to the above equations, we have
\begin{equation*}
\begin{split}
&\arg\left\{\frac{\left(\frac{137}{60}-5z+5z^2-\frac{10}{3}z^3+\frac{5}{4}z^4-\frac{1}{5}z^5\right)^{\alpha}}{\left(1-\frac{1}{2}z\right)^2}\right\}\\
&=\arg\left\{\left(1-z\right)^{\alpha}\right\}+\arg\left\{\left(\frac{137}{60}-\frac{163}{60}z+\frac{137}{60}z^2-\frac{63}{60}z^3+\frac{1}{5}z^4\right)^{\alpha}\right\}\\
&\quad+\arg\left\{\frac{1}{\left(1-\frac{1}{2}z\right)^2}\right\}=\alpha \theta_1+\alpha \theta_2+\theta_3.
\end{split}
\end{equation*}
Since $\alpha \theta_1+\alpha \theta_2+\theta_3\leqslant \theta_3\leqslant  \frac{\pi}{2}$ and $\alpha \theta_1+\alpha \theta_2+\theta_3\geqslant \theta_1+\theta_2+\theta_3$.
Next we need to prove $g(x)=\left(\theta_1+\theta_2+\theta_3\right)(x)\geqslant -\frac{\pi}{2}$.
With $y:=\cos x,$ we have
\begin{equation*}
\begin{split}
g'(x)=\frac{4}{\left(6576y^4-21174y^3+24106y^2-11144y+2536\right)\left(5-4y\right)}h(y).
\end{split}
\end{equation*}
Here $h(y)=-9864y^5+42348y^4-66272y^3+43988y^2-8407y-1343,~y\in(-1,1)$ and the roots are
$y_1\approx -0.0996$ with $x_1\approx 1.6705$, $y_2\approx 0.6531$ with $x_2\approx 0.8591$, respectively.
In further, we have $h(y)>0$ if $y\in(-1,y_1)$ and $h(y)<0$ if  $y\in(y_1,y_2)$ and $h(y)>0$ if  $y\in(y_2,1)$.
Moreover, combining with
$$\left(6576y^4-21174y^3+24106y^2-11144y+2536\right)\left(5-4y\right)>0,$$  it implies  that
 $g'(x)>0$ if $x\in(x_1,\pi)$ and $g'(x)<0$ if  $x\in(x_2,x_1)$ and $g'(x)>0$ if  $x\in(0,x_2)$.
Therefore,  the function $g$ attains its minimum at $x^\star=x_1$ 
and
\[g(x^\star)>-1.33>-\frac{\pi}{2}.\]
On the other hand, it can be easily checked that $g(0)=-\frac{\pi}{2}$ and $g(\pi)=0$.
Hence, we have $g(x)\geqslant-\frac{\pi}{2}$.
The proof is completed.
\end{proof}

\begin{lemma}(BDF6) \label{lemma4.003}
Let $q_j^k$ be defined by (\ref{2.8}). Then for any positive integer $N$, it holds that
\begin{equation*}
\begin{split}
\sum_{n=1}^N\left(\sum_{j=0}^{n-1} q_j^k v^{n-j},v^n\right)\geqslant 0.
\end{split}
\end{equation*}
\end{lemma}
\begin{proof}
From \eqref{2.8} and \eqref{3.001}, we get
$
q(\zeta)=\frac{g_6(\zeta)}{\mu_6(\zeta)}
$
with
\begin{equation*}
  \begin{split}
  g_6(\zeta)=&  \Bigg(\frac{147}{60}-6e^{-\sigma\tau}\zeta+\frac{15}{2}(e^{-\sigma\tau}\zeta)^2
-\frac{20}{3}(e^{-\sigma\tau}\zeta)^3+\frac{15}{4}(e^{-\sigma\tau}\zeta)^4\\
&-\frac{6}{5}(e^{-\sigma\tau}\zeta)^5
+\frac{1}{6}(e^{-\sigma\tau}\zeta)^6\Bigg)^{\alpha},
  \end{split}
\end{equation*}
and
$\mu_6(\zeta)=\left(1-\frac{3}{5}e^{-\sigma\tau}\zeta\right)\left(1-\frac{1}{2}e^{-\sigma\tau}\zeta\right)\left(1-\frac{1}{3}e^{-\sigma\tau}\zeta\right)$.

Taking $z=e^{-\sigma\tau}\zeta$, it leads to
\begin{equation*}
\begin{split}
&\frac{\left(\frac{147}{60}-6z+\frac{15}{2}z^2-\frac{20}{3}z^3+\frac{15}{4}z^4-\frac{6}{5}z^5+\frac{1}{6}z^6\right)^{\alpha}}{\left(1-\frac{3}{5}z\right)\left(1-\frac{1}{2}z\right)\left(1-\frac{1}{3}z\right)}\\
&=\frac{\left(1-z\right)^{\alpha}\left(\frac{147}{60}-\frac{213}{60}z+\frac{237}{60}z^2-\frac{163}{60}z^3+\frac{62}{60}z^4-\frac{10}{60}z^5\right)^{\alpha}}
{\left(1-\frac{3}{5}z\right)\left(1-\frac{1}{2}z\right)\left(1-\frac{1}{3}z\right)}.
\end{split}
\end{equation*}

Next we apply  the  Grenander-Szeg\"{o} theorem  to obtain the desired result.
Let  $z=e^{ix}$ with $x\in [0,\pi]$, we have
\begin{equation*}
\begin{split}
\left(1-z\right)^{\alpha}=\left(2\sin\frac{x}{2}\right)^{\alpha}e^{i{\alpha}\theta_1}
\end{split}
\end{equation*}
with $\theta_1=\arctan\frac{-\sin (x)}{1-\cos x}=\frac{x-\pi}{2}\leqslant 0$; and
\begin{equation*}
\begin{split}
\left(\frac{147}{60}-\frac{213}{60}z+\frac{237}{60}z^2-\frac{163}{60}z^3+\frac{62}{60}z^4-\frac{10}{60}z^5\right)^{\alpha}
=\left(a_6-ib_6\right)^{\alpha}=\left(a_6^2+b_6^2\right)^{\frac{\alpha}{2}}e^{i{\alpha}\theta_2}
\end{split}
\end{equation*}
with
\begin{equation*}
\begin{split}
a_6(x)&=\frac{1}{60}\left(147-213\cos x+237\cos(2x)-163\cos(3x)+62\cos(4x)-10\cos(5x)\right),\\
b_6(x)&=\frac{1}{60}\left(213\sin x-237\sin(2x)+163\sin(3x)-62\sin(4x)+10\sin(5x)\right)\geqslant 0,\\
\end{split}
\end{equation*}
and $\theta_2=\arctan\frac{-b_6(x)}{a_6(x)}\leqslant 0,~~a_6(x)\geqslant 0,$ or $\theta_2=\arctan\frac{-b_6(x)}{a_6(x)}-\pi\leqslant 0,~~a_6(x)\leqslant 0.$
Furthermore, there exists
\begin{equation*}
\begin{split}
\frac{1}{1-\frac{3}{5}z}=\frac{1}{\sqrt{\frac{34}{25}-\frac{6}{5}\cos x}}e^{i\theta_3},~~\theta_3=\arctan\frac{\frac{3}{5}\sin x}{1-\frac{3}{5}\cos x}\geqslant 0,
\end{split}
\end{equation*}
\begin{equation*}
\begin{split}
\frac{1}{1-\frac{1}{2}z}=\frac{1}{\sqrt{\frac{5}{4}-\cos x}}e^{i\theta_4},~~\theta_4=\arctan\frac{\frac{1}{2}\sin x}{1-\frac{1}{2}\cos x}\geqslant 0,
\end{split}
\end{equation*}
\begin{equation*}
\begin{split}
\frac{1}{1-\frac{1}{3}z}=\frac{1}{\sqrt{\frac{10}{9}-\frac{2}{3}\cos x}}e^{i\theta_5},~~\theta_5=\arctan\frac{\frac{1}{3}\sin x}{1-\frac{1}{3}\cos x}\geqslant 0.
\end{split}
\end{equation*}
From Lemma \ref{lemma2.7}, we need to prove
\begin{equation*}
\begin{split}
\Real\left\{\frac{\left(\frac{147}{60}-6z+\frac{15}{2}z^2-\frac{20}{3}z^3+\frac{15}{4}z^4-\frac{6}{5}z^5+\frac{1}{6}z^6\right)^{\alpha}}{\left(1-\frac{3}{5}z\right)\left(1-\frac{1}{2}z\right)\left(1-\frac{1}{3}z\right)}\right\}\geqslant0,
\end{split}
\end{equation*}
which is equal to prove
\begin{equation*}
\begin{split}
\arg\left\{\frac{\left(\frac{147}{60}-6z+\frac{15}{2}z^2-\frac{20}{3}z^3+\frac{15}{4}z^4-\frac{6}{5}z^5+\frac{1}{6}z^6\right)^{\alpha}}{\left(1-\frac{3}{5}z\right)\left(1-\frac{1}{2}z\right)\left(1-\frac{1}{3}z\right)}\right\}\in\left[-\frac{\pi}{2},\frac{\pi}{2}\right].
\end{split}
\end{equation*}
According to the above equations, we have
\begin{equation*}
\begin{split}
&\arg\left\{\frac{\left(\frac{147}{60}-6z+\frac{15}{2}z^2-\frac{20}{3}z^3+\frac{15}{4}z^4-\frac{6}{5}z^5+\frac{1}{6}z^6\right)^{\alpha}}{\left(1-\frac{3}{5}z\right)\left(1-\frac{1}{2}z\right)\left(1-\frac{1}{3}z\right)}\right\}\\
&=\arg\left\{\left(1-z\right)^{\alpha}\right\}+\arg\left\{\left(\frac{147}{60}-\frac{213}{60}z+\frac{237}{60}z^2-\frac{163}{60}z^3+\frac{62}{60}z^4-\frac{10}{60}z^5\right)^{\alpha}\right\}\\
&\quad+\arg\left\{\frac{1}{1-\frac{3}{5}z}\right\}+\arg\left\{\frac{1}{1-\frac{1}{2}z}\right\}+\arg\left\{\frac{1}{1-\frac{1}{3}z}\right\}\\
&=\alpha\theta_1+\alpha\theta_2+\theta_3+\theta_4+\theta_5.
\end{split}
\end{equation*}
Since $\alpha \theta_1+\alpha \theta_2+\theta_3+\theta_4+\theta_5\geqslant \theta_1+\theta_2+\theta_3+\theta_4+\theta_5$ and $\alpha \theta_1+\alpha \theta_2+\theta_3+\theta_4+\theta_5\leqslant \theta_3+\theta_4+\theta_5<\frac{\pi}{2}$. In fact, $\delta(x)=\theta_3+\theta_4+\theta_5<\frac{\pi}{2}$ follows by taking $\cos x=y$, we have
\begin{equation*}
\begin{split}
\delta'(x)=\frac{2}{\left(17-15y\right)\left(5-4y\right)\left(5-3y\right)}p(y).
\end{split}
\end{equation*}
Here $p(y)=135y^3-429y^2+420y-120,~y\in(-1,1)$
has the root $y_1\approx0.5041$ with $x_1\approx1.0425$.
In further, we obtain $p(y)<0$ if $y\in(-1,y_1)$ and $p(y)>0$ if $y\in(y_1,1)$.
 Moreover, combining with $\left(17-15y\right)\left(5-4y\right)\left(5-3y\right)>0$, it implies  that
 $\delta'(x)<0$ if $x\in(x_1,\pi)$ and $\delta'(x)>0$ if $x\in(0,x_1)$.
  Therefore, the function  $\delta$ attains its maximum at $x^\star=x_1$ 
and
\[\delta(x^\star)<1.5<\frac{\pi}{2}.\]
Next we just need to prove $g(x)=\theta_1+\theta_2+\theta_3+\theta_4+\theta_5\geqslant -\frac{\pi}{2}$.
With $y:=\cos x,$ we obtain $g'(x)=4h(y)/C_y.$
Here
\begin{equation*}
\begin{split}
C_y
=&\left(-11760y^5\!+\!44976y^4\!-\!64374y^3\!+\!40906y^2\!-\!9944y\!+\!1096\right)\\
&\times\left(5-4y\right)\left(5-3y\right)\left(17-15y\right)>0,~~\forall y\in [-1,1];
\end{split}
\end{equation*}
and
\begin{equation*}
\begin{split}
 h(y)=&793800y^8-6314580y^7+20885463y^6- 37146627y^5+38067828y^4\\
      &- 21920022y^3+5908998y^2- 72525y-199635,~y\in[-1,1],
\end{split}
\end{equation*}
 which  has the roots
$y_1\approx -0.1391$ with $x_1\approx 1.7103$, $y_2\approx 0.5015$ with $x_2\approx 1.0455$.
In further, we have $h(y)>0$ if $y\in(-1,y_1)$ and $h(y)<0$ if  $y\in(y_1,y_2)$ and $h(y)>0$ if  $y\in(y_2,1)$.
It implies  that
 $g'(x)>0$ if $x\in(x_1,\pi)$ and $g'(x)<0$ if  $x\in(x_2,x_1)$ and $g'(x)>0$ if  $x\in(0,x_2)$.
Therefore, the funciton  $g$ attains its minimum at $x^\star=x_1$ with $a_6(x_1)<0$ 
and
\[g(x^\star)>-1.566>-\frac{\pi}{2}.\]
On the other hand, it can be easily checked that $g(0)=-\frac{\pi}{2}$ and $g(\pi)=0$.
Hence, we have $g(x)\geqslant-\frac{\pi}{2}$.
The proof is completed.
\end{proof}
\begin{remark}
  From \eqref{3.001} and  Lemmas \ref{lemma4.1}-\ref{lemma4.003}, it yields
  \begin{equation*}
\max\arg \left(q(\zeta)\right)\leqslant \frac{\pi}{2},~\alpha \in [0,1],~k=3,4,5,6.
\end{equation*}
Figure \ref{Fig.5.1} also shows that the argument of $q(\zeta)$ are less than or equal to $\frac{\pi}{2}$.
\end{remark}
\begin{figure}[t] \centering
  \begin{tabular}{cc}
      \includegraphics[width=0.5\textwidth]{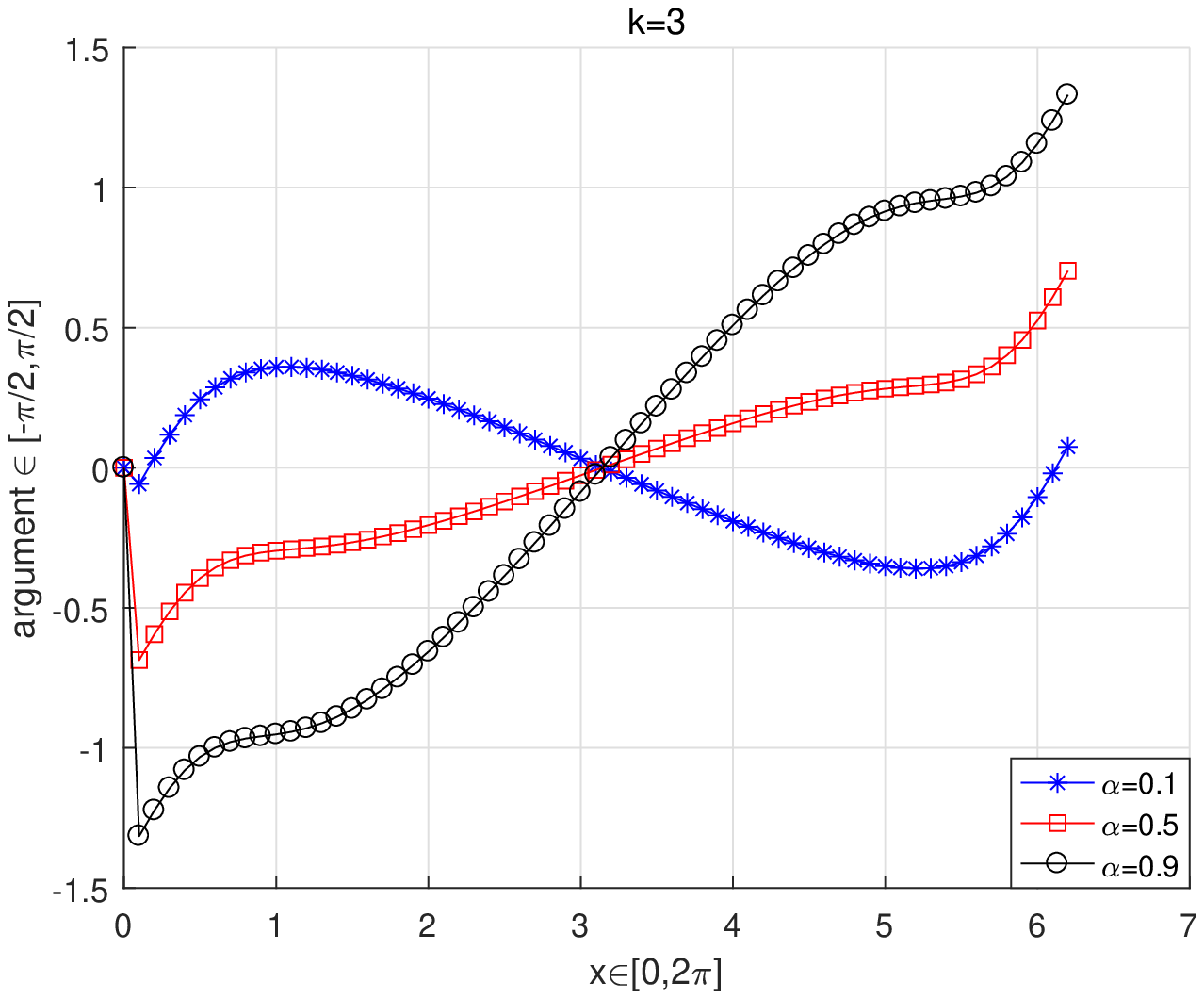}&\includegraphics[width=0.5\textwidth]{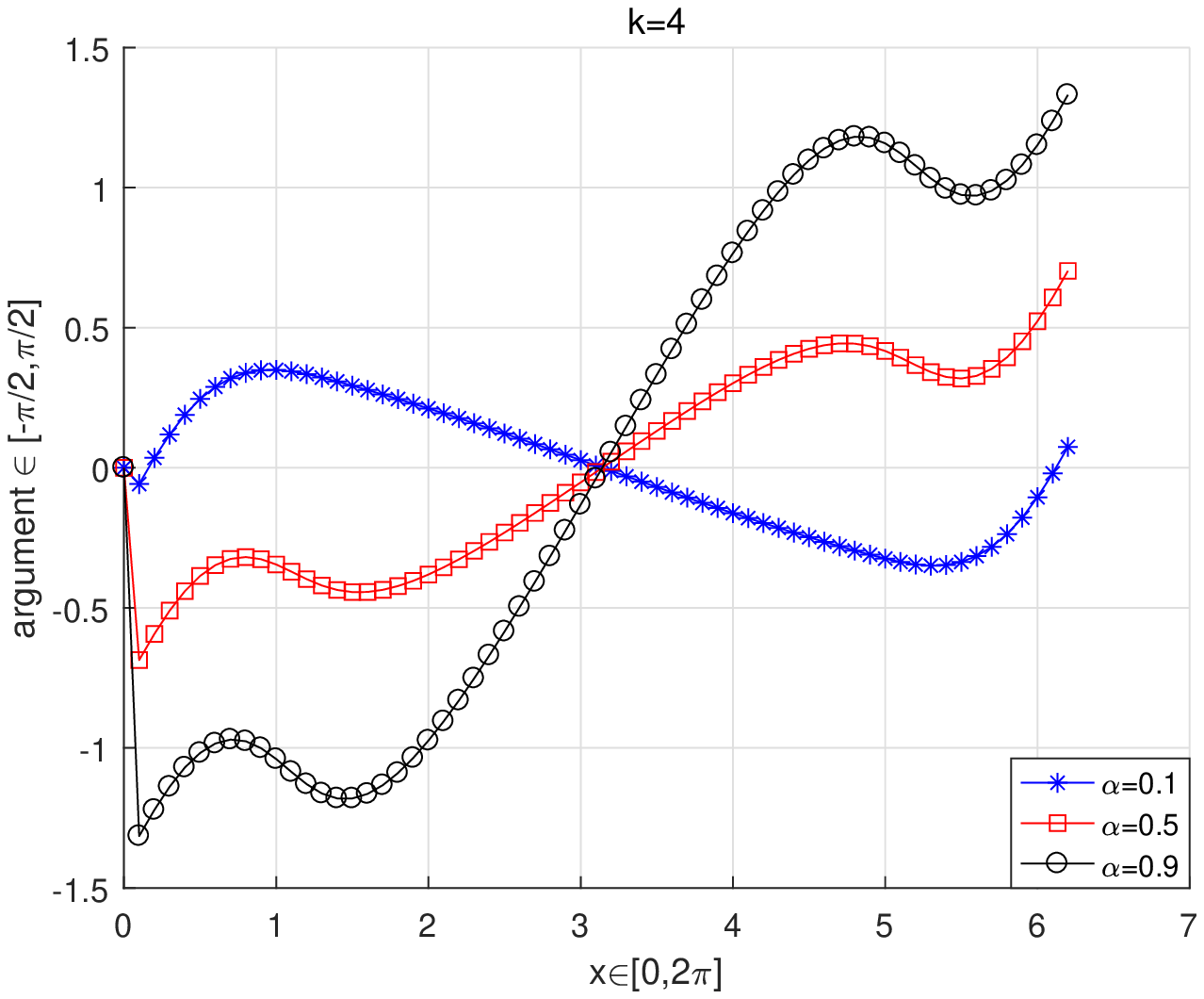}\\
      \includegraphics[width=0.5\textwidth]{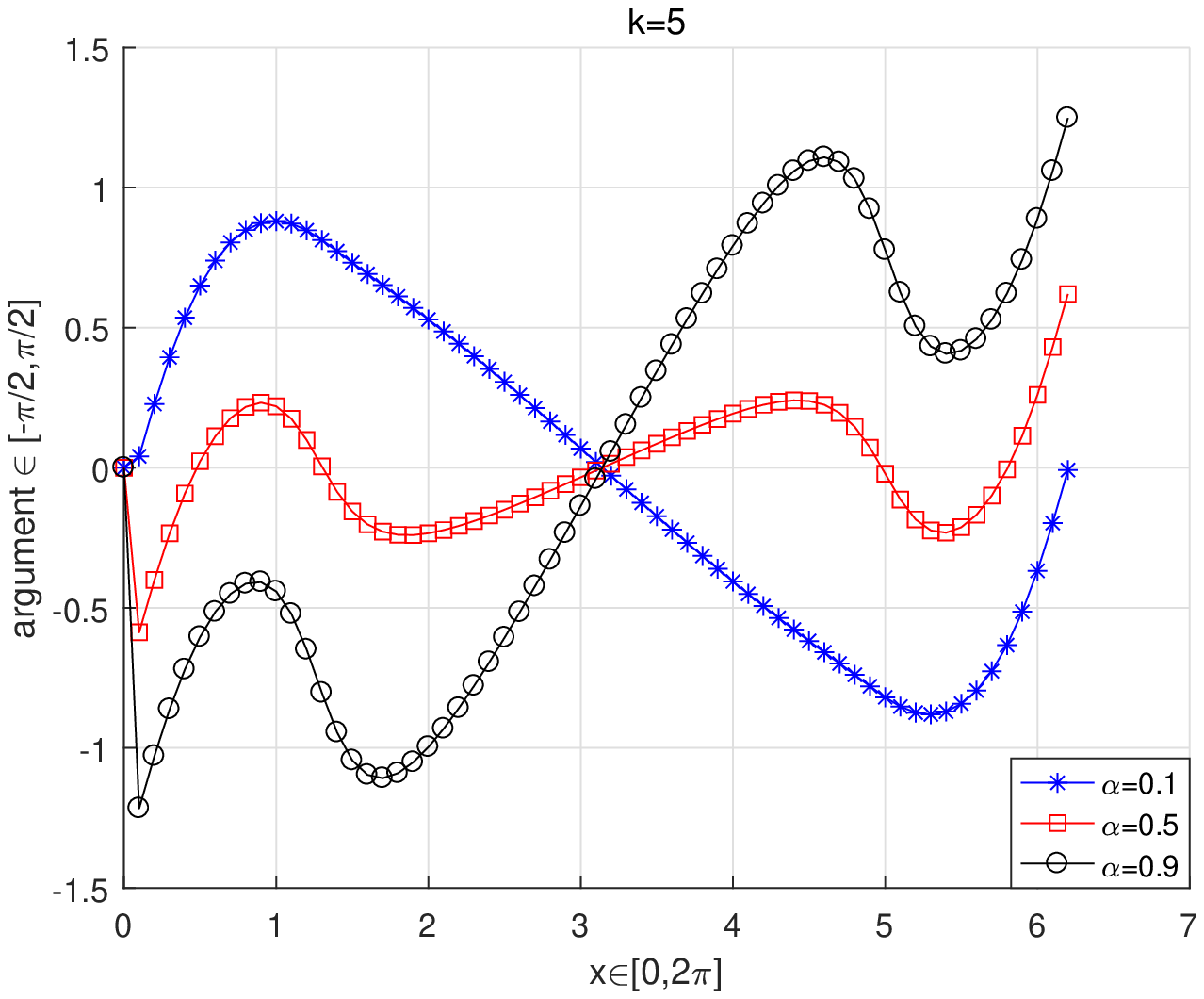}&\includegraphics[width=0.5\textwidth]{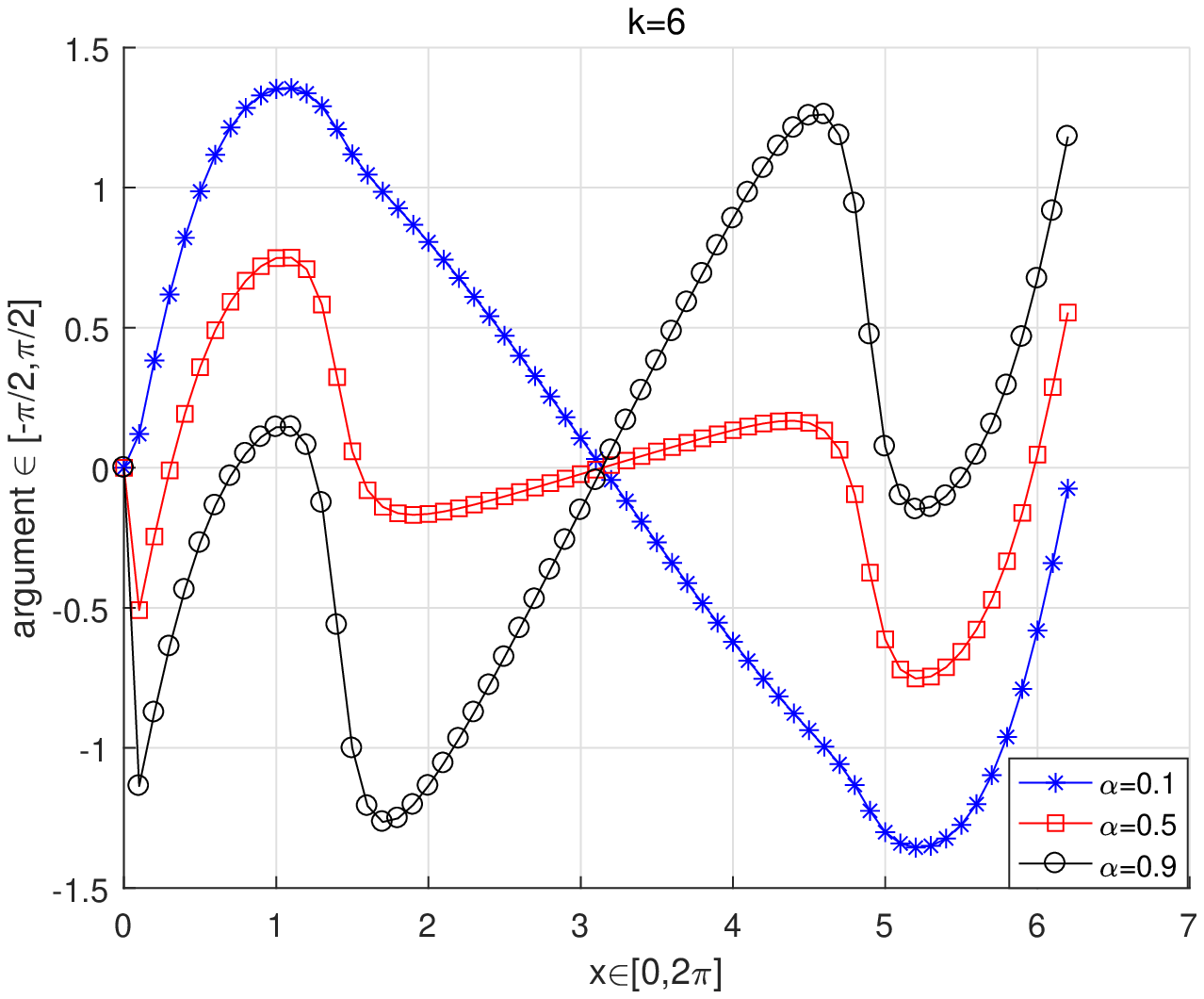}
  \end{tabular}
  \caption{Argument of generating power series $q(\zeta)$ with $\sigma=0$ in \eqref{3.001}.}
  \label{Fig.5.1}
\end{figure}

\section{Stability analysis}\label{Se:stab}

In this section we prove stability of the up to  six-step BDF method \eqref{2.0083} by the energy
technique for subdiffusion equation.  The novelty is the simplicity of the proof,
the  main advantage of the energy technique. The result is well known if $\alpha=1$.
Proofs by other stability techniques are significantly more involved. For example,
by a spectral technique in the case of selfadjoint operators, we refer to \cite[chapter 10]{Thomee:06};
for a proof in the general case, under a sharp condition on the nonselfadjointness of the operator
as well as for nonlinear parabolic equations,
by a combination of spectral and Fourier techniques, see, e.g., \cite{Akrivis:18} and references therein.
For a long-time estimate in the case of selfadjoint operators and an application to the Stokes--Darcy problem, see \cite{LiWangZhou:2020}.
If $0<\alpha<1$, by operational calculus in the case of selfadjoint operators, we refer to \cite{Jin:16,Jin:17}.

\begin{theorem} \label{theorem4.3}	
Let $\tilde{u}^n$ be the approximate solution of $u^n$, which is the exact solution of BDF$k$ corrected scheme (\ref{2.0083}) with $\epsilon^n=\tilde{u}^n-u^n$. Then the BDF$k$  corrected scheme (\ref{2.0083})
  with $k=3,4,5,6$ is stable in the sense that
\begin{equation*}
\begin{split}
\tau\sum_{n=1}^N\|\epsilon^n\|^2\leqslant C \tau\sum_{n=1}^N\| \epsilon^0\|^2~~{\rm and}~~\tau\sum_{n=1}^N\| \epsilon^n\|\leqslant C T\| \epsilon^0\|.
\end{split}
\end{equation*}
\end{theorem}
\begin{proof}
Let $\tilde{u}^n$ be the approximate solution of $u^n$, which is the exact solution of BDF$k$ corrected scheme (\ref{2.0083}) with $k=3,4,5,6$.
Putting $\epsilon^n=\tilde{u}^n-u^n$, we have
\begin{equation*}
P\left(\bar{\partial}_{\tau}\right)(\epsilon^n-e^{-\sigma n \tau}\epsilon^0)+A\epsilon^n=-e^{-\sigma n \tau}a_n^{(k)}A\epsilon^0.
\end{equation*}
Let  $\eta^n=\epsilon^n-e^{-\sigma n \tau}\epsilon^0$ with $\eta^0=0$, it yields
\begin{equation}\label{4.10}
P\left(\bar{\partial}_{\tau}\right)\eta^n+A\eta^n=-e^{-\sigma n \tau}\left(1+a_n^{(k)}\right)A\epsilon^0.
\end{equation}
Taking in \eqref{4.10} the inner product with
$v^n=\eta^{n}-\mu_1e^{-\sigma \tau}\eta^{n-1}-\dotsb-\mu_ke^{-\sigma k \tau}\eta^{n-k}$, we have
\begin{equation*}
\left( P\left(\bar{\partial}_{\tau}\right)\eta^n,v^n\right)+\left\langle\eta^n,v^n\right\rangle=-e^{-\sigma n \tau}\left(1+a_n^{(k)}\right)\left\langle\epsilon^0,v^n\right\rangle\leqslant\left|1+a_n^{(k)}\right|\|\epsilon^0\|\cdot\|v^n\|.
\end{equation*}
Multiplying the above inequality by $\tau$ and summing up for $n$ from $1$ to $N$, we get
\begin{equation*}
\begin{split}
\tau\sum_{n=1}^N\left( P\left(\bar{\partial}_{\tau}\right)\eta^n,v^n\right)+\tau\sum_{n=1}^N\left\langle\eta^n,v^n\right\rangle
\leqslant C\tau\sum_{n=1}^N \|\epsilon^0\|\cdot\| v^n\|.
\end{split}
\end{equation*}
According to \eqref{2.9}, Lemmas \ref{lemma4.1}-\ref{lemma4.003} and the above inequality, it implies that
\begin{equation}\label{4.11}
\begin{split}
\tau\sum_{n=1}^N\left\langle\eta^n,v^n\right\rangle\leqslant C \tau\sum_{n=1}^N\left(\frac{\|\epsilon^0\|^2} {4\varepsilon} +\varepsilon\| v^n\|^2\right)~~\forall \varepsilon>0.
\end{split}
\end{equation}
Next we  prove  the following inequality \eqref{4.12} for  three cases: $k=3,4$; $k=5$; and $k=6$.
\begin{equation}\label{4.12}
\begin{split}
\tau\sum_{n=1}^N\|\eta^n\|^2\leqslant C \tau\sum_{n=1}^N\| \epsilon^0\|^2.
\end{split}
\end{equation}

Case $1$: Let $k=3,4$ with $v^n=\eta^n-\frac{1}{2}e^{-\sigma \tau}\eta^{n-1}$. Using  \eqref{4.11}, Lemma \ref{lemma3.1}, we have
\begin{equation*}
\begin{split}
\frac{1}{2}\tau\sum_{n=1}^N\|\eta^n\|^2\leqslant C \tau\sum_{n=1}^N\left(\frac{\|\epsilon^0\|^2} {4\varepsilon_1}+\varepsilon_1\|v^n\|^2\right)\leqslant C \tau\sum_{n=1}^N\frac{\|\epsilon^0\|^2} {4\varepsilon_1} +\frac{5}{2}C\varepsilon_1\tau\sum_{n=1}^N\| \eta^n\|^2.
\end{split}
\end{equation*}
By choosing a sufficiently small $\varepsilon_1$, the desired result (\ref{4.12}) is obtained.

Case $2$:  Let $k=5$ with  $v^n=\eta^n-e^{-\sigma\tau}\eta^{n-1}+\frac{1}{4}e^{-2\sigma\tau}\eta^{n-2}$. According to  \eqref{4.11}, Lemma \ref{lemma3.1}, we obtain
\begin{equation*}
\begin{split}
\frac{1}{4}\tau\sum_{n=1}^N\|\eta^n\|^2\leqslant C \tau\sum_{n=1}^N\left(\frac{\|\epsilon^0\|^2} {4\varepsilon_2}+\varepsilon_2\|v^n\|^2\right)\leqslant C \tau\sum_{n=1}^N\frac{\|\epsilon^0\|^2} {4\varepsilon_2} +\frac{25}{4}C \varepsilon_2\tau\sum_{n=1}^N\| \eta^n\|^2.
\end{split}
\end{equation*}
By choosing sufficiently small $\varepsilon_2$, the desired result (\ref{4.12}) is obtained.

Case $3$:  Let $k=6$ with  $v^n=\eta^n-\frac{43}{30}e^{-\sigma\tau}\eta^{n-1}+\frac{2}{3}e^{-2\sigma \tau}\eta^{n-2}-\frac{1}{10}e^{-3\sigma \tau}\eta^{n-3}$. From \eqref{4.11}, Lemma \ref{lemma3.1}, it yields
\begin{equation*}
\begin{split}
\frac{1}{24}\tau\sum_{n=1}^N\|\eta^n\|^2\leqslant C \tau\sum_{n=1}^N\left(\frac{\|\epsilon^0\|^2} {4\varepsilon_3}+\varepsilon_3\|v^n\|^2\right)\leqslant C \tau\sum_{n=1}^N\frac{\|\epsilon^0\|^2} {4\varepsilon_3} +\frac{44}{3}C \varepsilon_3\tau\sum_{n=1}^N\| \eta^n\|^2.
\end{split}
\end{equation*}
By choosing a sufficiently small $\varepsilon_3$, the desired result (\ref{4.12}) is obtained.

On the one hand, there exists
\begin{equation*}
\begin{split}
\| \epsilon^n\|^2=\|e^{-\sigma n \tau}\epsilon^0+ \epsilon^n- e^{-\sigma n \tau}\epsilon^0\|^2\leqslant 2\left(\| \epsilon^0\|^2+\| \eta^n\|^2\right).
\end{split}
\end{equation*}
From  (\ref{4.12}) and the above inequality, we get
\begin{equation*}
\begin{split}
\tau\sum_{n=1}^N\| \epsilon^n\|^2\leqslant 2\left(C+1\right) \tau\sum_{n=1}^N\| \epsilon^0\|^2.
\end{split}
\end{equation*}

On the other hand, using  Cauchy-Schwarz inequality and (\ref{4.12}), it yields
\begin{equation*}
\begin{split}
\left(\tau \sum_{n=1}^{N}\|\eta^n\|\right)^2\leqslant\left(\tau \sum_{n=1}^{N}1\right) \left(\tau \sum_{n=1}^{N}\| \eta^n\|^2 \right)\leqslant C T^2\| \epsilon^0\|^2,
\end{split}
\end{equation*}
which leads to
$$\tau \sum_{n=1}^{N}\| \epsilon^n\|-\tau \sum_{n=1}^{N}\|  e^{-\sigma n \tau}\epsilon^0\|\leqslant \tau \sum_{n=1}^{N}\| \epsilon^n- e^{-\sigma n \tau}\epsilon^0\|=\tau \sum_{n=1}^{N}\|\eta^n\|\leqslant C T\| \epsilon^0\|,$$
\begin{equation*}
i.e.,
\begin{split}
\tau \sum_{n=1}^{N}\|\epsilon^n\|\leqslant \left(C+1\right) T\|\epsilon^0\|.
\end{split}
\end{equation*}
The proof is completed.
\end{proof}

\section*{Appendix}
Let \begin{equation*}
g(\zeta)=\left(\sum_{j=1}^{k}\frac{1}{j}(1-e^{-\sigma\tau}\zeta)^{j}\right)^\alpha=\sum_{j=0}^{\infty}g_j^k{\zeta}^{j},\quad g_j^k=e^{-\sigma j \tau}l_j^k.
\end{equation*}
Then the coefficients $\{l_j^k\}_{j=0}^{\infty}$ are given explicitly by the following  recurrence relation, see Theorem 1.6 of \cite{Henrici:74}.

\begin{itemize}
\item BDF1
 $$l_0^k=1,~l_j^k=\left(1-\frac{{\alpha}+1}{j}\right)l_{j-1}^k,~j\geqslant1.$$
\end{itemize}

\begin{itemize}
\item BDF2
\begin{equation*}
\begin{split}
l_0^k&=\left(\frac{3}{2}\right)^{\alpha},~l_1^k=-\left(\frac{3}{2}\right)^{\alpha}\frac{4}{3}{\alpha},\\
l_j^k&=\frac{4}{3}\left(1-\frac{{\alpha}+1}{j}\right)l_{j-1}^k+\frac{1}{3}\left(\frac{2({\alpha}+1)}{j}-1\right)l_{j-2}^k,~j\geqslant2.
\end{split}
\end{equation*}
\end{itemize}

\begin{itemize}
\item BDF3
\begin{equation*}
\begin{split}
l_0^k&=\left(\frac{11}{6}\right)^{\alpha},~l_1^k=-\left(\frac{11}{6}\right)^{\alpha}\frac{18}{11}{\alpha},~
l_2^k=\left(\frac{11}{6}\right)^{\alpha}\left(\frac{162}{121}{\alpha}^2-\frac{63}{121}{\alpha}\right),\\
l_j^k&=\frac{18}{11}\left(1-\frac{{\alpha}+1}{j}\right)l_{j-1}^k\!+\!\frac{18}{22}\left(\frac{2({\alpha}+1)}{j}-1\right)\!l_{j-2}^k\\
&\quad+\frac{2}{11}\left(1-\frac{3\left({\alpha}+1\right)}{j}\right)\!l_{j-3}^k,~j\geqslant3.
\end{split}
\end{equation*}
\end{itemize}

\begin{itemize}
\item BDF4
\begin{equation*}
\begin{split}
l_0^k&=\left(\frac{25}{12}\right)^{\alpha},~l_1^k=-\left(\frac{25}{12}\right)^{\alpha}\frac{48}{25}{\alpha},~
l_2^k=\left(\frac{25}{12}\right)^{\alpha}\left(\frac{1152}{625}{\alpha}^2-\frac{252}{625}{\alpha}\right),\\
l_3^k&=\left(\frac{25}{12}\right)^{\alpha}\left(-\frac{18432}{15625}{\alpha}^3+\frac{12096}{15625}{\alpha}^2-\frac{3664}{15625}{\alpha}\right),\\
l_j^k&=\frac{48}{25}\left(1-\frac{{\alpha}+1}{j}\right)l_{j-1}^k+\frac{36}{25}\left(\frac{2({\alpha}+1)}{j}-1\right)l_{j-2}^k\\
&\quad+\frac{16}{25}\left(1-\frac{3\left({\alpha}+1\right)}{j}\right)l_{j-3}^k+\frac{3}{25}\left(\frac{4({\alpha}+1)}{j}-1\right)l_{j-4}^k,~j\geqslant4.
\end{split}
\end{equation*}
\end{itemize}

\begin{itemize}
\item BDF5
\begin{equation*}
\begin{split}
l_0^k&=\left(\frac{137}{60}\right)^{\alpha},~l_1^k=-\left(\frac{137}{60}\right)^{\alpha}\frac{300}{137}{\alpha},~
l_2^k=\left(\frac{137}{60}\right)^{\alpha}\left(\frac{45000}{18769}{\alpha}^2-\frac{3900}{18769}{\alpha}\right),\\
l_3^k&=\left(\frac{137}{60}\right)^{\alpha}\left(-\frac{4500000}{2571353}{\alpha}^3+\frac{1170000}{2571353}{\alpha}^2-\frac{423800}{2571353}{\alpha}\right),\\
l_4^k&=\left(\frac{137}{60}\right)^{\alpha}\left(\frac{337500000}{352275361}{\alpha}^4-\frac{175500000}{352275361}{\alpha}^3+\frac{134745000}{352275361}{\alpha}^2-\frac{103893525}{352275361}{\alpha}\right),\\
l_j^k&=\!\frac{300}{137}\left(1\!-\!\frac{{\alpha}+1}{j}\right)\!l_{j-1}^k\!+\!\frac{300}{137}\left(\frac{2({\alpha}+1)}{j}\!-1\!\right)\!l_{j-2}^k
\!+\!\frac{200}{137}\left(1\!-\!\frac{3\left({\alpha}\!+\!1\right)}{j}\right)\!l_{j-3}^k\\
&\quad+\frac{75}{137}\left(\frac{4({\alpha}+1)}{j}-1\right)l_{j-4}^k+\frac{12}{137}\left(1-\frac{5\left({\alpha}+1\right)}{j}\right)l_{j-5}^k,~j\geqslant5.
\end{split}
\end{equation*}
\end{itemize}

\begin{itemize}
\item BDF6
\begin{equation*}
\begin{split}
l_0^k&=\left(\frac{147}{60}\right)^{\alpha},~l_1^k=-\left(\frac{147}{60}\right)^{\alpha}\frac{360}{147}{\alpha},~
l_2^k=\left(\frac{147}{60}\right)^{\alpha}\left(\frac{7200}{2401}{\alpha}^2+\frac{150}{2401}{\alpha}\right),\\
l_3^k&=\left(\frac{147}{60}\right)^{\alpha}\left(-\frac{288000}{117649}{\alpha}^3-\frac{18000}{117649}{\alpha}^2-\frac{42400}{352947}{\alpha}\right),\\
l_4^k&=\left(\frac{147}{60}\right)^{\alpha}\left(\frac{8640000}{5764801}{\alpha}^4+\frac{1080000}{5764801}{\alpha}^3
+\frac{1707250}{5764801}{\alpha}^2-\frac{2603575}{5764801}{\alpha}\right),\\
l_5^k&=\left(\frac{147}{60}\right)^{\alpha}\left(-\frac{207360000}{282475249}{\alpha}^5-\frac{43200000}{282475249}{\alpha}^4\right.
-\frac{14730000}{40353607}{\alpha}^3+\frac{310309000}{282475249}{\alpha}^2\\
&\quad\quad\quad\quad\quad\quad\left.-\frac{94994224}{282475249}{\alpha}\right),\\
l_j^k&=\frac{360}{147}\left(1\!-\!\frac{{\alpha}+1}{j}\right)\!l_{j-1}^k\!+\!\frac{450}{147}\left(\frac{2({\alpha}+1)}{j}\!-\!1\right)\!l_{j-2}^k
\!+\!\frac{400}{147}\left(1\!-\!\frac{3\left({\alpha}+1\right)}{j}\right)\!l_{j-3}^k\\
&\quad+\frac{225}{147}\left(\frac{4({\alpha}+1)}{j}-1\right)l_{j-4}^k+\frac{72}{147}\left(1-\frac{5\left({\alpha}+1\right)}{j}\right)l_{j-5}^k\\
&\quad+\frac{10}{147}\left(\frac{6({\alpha}+1)}{j}-1\right)l_{j-6}^k,~j\geqslant6.
\end{split}
\end{equation*}
\end{itemize}

\bibliographystyle{amsplain}

\begin{thebibliography}{10}


\bibitem{Akrivis:18}
Akrivis, G.: Stability of implicit and implicit--explicit multistep methods for nonlinear parabolic equations. IMA J. Numer. Anal. \textbf{38}, 1768--1796 (2018)


\bibitem{ACYZ:20}
Akrivis, G., Chen, M.H., Yu, F., Zhou, Z.: The energy technique for the six-step BDF method. arXiv:2007.08924, Math. Comp. Revised.


\bibitem{AK:16}
Akrivis, G.,  Katsoprinakis, E.: Backward difference formulae: New multipliers and stability properties for parabolic equations.
Math. Comp. \textbf{85},  2195--2216 (2016)



 \bibitem{BC:89}
Baiocchi, C.,  Crouzeix,  M.: On the equivalence of A-stability and G-stability. Appl. Numer. Math. \textbf{5},  19--22 (1989)


\bibitem{CB:11} {Carmi, S., Barkai, E.:}
{Fractional Feynman-Kac equation for weak ergodicity breaking}.
Phys. Rev. E \textbf{84}, 061104 (2011)

\bibitem{Carmi:10} {Carmi, S.,  Turgeman, L., Barkai, E.:}
{On distributions of functionals of anomalous diffusion paths}.
J. Stat. Phys. \textbf{141}, 1071--1092 (2010)


\bibitem{Chan:07}   Chan, R.H.F.,  Jin, X.Q.:  An Introduction to Iterative Toeplitz Solvers. SIAM, Philadelphia (2007)




\bibitem{CD:13}   Chen, M.H., Deng, W.H.:  Fourth order difference approximations for space Riemann-Liouville derivatives based on weighted and shifted Lubich difference operators.
Commun. Comput. Phys. \textbf{16}, 516--540 (2014)


\bibitem{CD:14}  Chen, M.H., Deng, W.H.:  Fourth order accurate scheme for the space fractional diffusion equations.
SIAM J. Numer. Anal.  \textbf{52}, 1418--1438 (2014)


\bibitem{Chendeng:13} Chen, M.H., Deng, W.H.:    Discretized fractional substantial calculus.
ESAIM: M2AN. \textbf{49},  373--394 (2015)

\bibitem{ChenD:15} {  Chen, M.H., Deng, W.H.:}
{High order algorithms for the fractional substantial diffusion equation with truncated L\'{e}vy flights}.
SIAM J. Sci. Comput. \textbf{37}, A890--A917 (2015)

\bibitem{CD:18}  Chen, M.H., Deng, W.H.:  High order algorithm for the time-tempered fractional Feynman-Kac equation.
J. Sci. Comput.  \textbf{76}, 867--887  (2018)


\bibitem{Cuesta:06}  Cuesta, E.,     Lubich, C.,  Palencia, C.: Convolution quadrature time discretization of fractional diffusion-wave equations. Math. Comp. \textbf{75}, 673--696 (2006)

\bibitem{Dahlquist:78}
Dahlquist, G.: G-stability is equivalent to A-stability. BIT \textbf{18},  384--401 (1978)

\bibitem{Gao:15}  Gao, G.H., Sun, H.H., Sun, Z.Z.:  Stability and convergence of finite difference schemes for a class of time-fractional sub-diffusion equation based on certain superconvergence. J. Comput. Phys. \textbf{280}, 510--528 (2015)


\bibitem{FJBE:06} {Friedrich, R., Jenko, F., Baule, A., Eule, S.:}
{Anomalous diffusion of inertial, weakly damped particles}.
Phys. Rev. Lett. \textbf{96}, 230601 (2006)

\bibitem{HW:10}
Hairer, E.,   Wanner, G.: Solving Ordinary Differential Equations II: Stiff and Differential--Algebraic Problems,  2nd ed.  Springer, Berlin (2010)






\bibitem{Henrici:74}  Henrici, P.:    Applied and Computational Complex Analysis.  New York (1974)


\bibitem{Ji:15} Ji, C.C.,  Sun, Z.Z.:  A high-order compact finite difference schemes for the fractional sub-diffusion equation.
J. Sci. Comput.  \textbf{64}, 959--985 (2015)


\bibitem{Jin:16}   Jin, B.T.,  Lazarov, R., Zhou, Z.:  Two fully discrete schemes for fractional diffusion and diffusion-wave equations with nonsmooth data. SIAM J. Sci. Comput. \textbf{38},  A146--A170 (2016)

\bibitem{JLTZ:17}   Jin, B.T.,  Lazarov, Thom\'ee, T, R., Zhou, Z.: On nonnegativity preservation in finite element methods for subdiffusion equations. Math. Comput. \textbf{86},  2239--2260 (2017)


\bibitem{Jin:17}    Jin, B.T., Li, B.Y.,  Zhou, Z.:  Correction of high-order BDF convolution quadrature for fractional evolution equations. SIAM J. Sci. Comput. \textbf{39},  A3129--A3152 (2017)


\bibitem{Klafter:11}  Klafter, J.,   Sokolov, I.M.: First Steps in Random Walks: From Tools to Applications.
Oxford University Press, New York (2011)


\bibitem{LiWangZhou:2020}
Li, B.Y.,   Wang, K.,   Zhou, Z.: Long-time accurate symmetrized implicit-explicit {BDF} methods for a class of parabolic equations with non-self-adjoint operators. SIAM J. Numer. Anal. \textbf{58},  189--210 (2020)

\bibitem{LiD:13}  Li, C.P., Ding, H.F.:  Higher order finite difference method for the reaction and anomalous-diffusion equation.
Appl. Math. Modell. \textbf{38}, 3802--3821 (2014)



\bibitem{Liu:13}  Liu, J.:  Simple and efficient ALE methods with provable temporal accuracy up to fifth order for the stokes equations on time varying domains.
SIAM J. Numer. Anal.  \textbf{51}, 743--772 (2013)

\bibitem{LMV:13}  Lubich, C.,  Mansour, D., Venkataraman, C.: Backward difference time discretization of parabolic differential equations on evolving surfaces. IMA J. Numer. Anal. \textbf{33}, 1365--1385 (2013)


\bibitem{Lubich:86}  Lubich, C.:  Discretized fractional calculus. SIAM J. Math. Anal. \textbf{17}, 704--719 (1986)


\bibitem{Lu:88} Lubich, C.:  Convolution quadrature and discretized operational calculus I.  Numer. Math. \textbf{52}, 129--145 (1988)


\bibitem{Lu:96}  Lubich, C, Sloan, I.H.,  Thom\'{e}e, V.:  Nonsmooth data error estimates for approximations of an evolution equation with a positive-type memory term. Math. Comp. \textbf{65}, 1--17 (1996)





\bibitem{NO:81}
 Nevanlinna, O.,  Odeh, F.: Multiplier techniques for linear multistep methods.
Numer. Funct. Anal. Optim. \textbf{3}, 377--423 (1981)



\bibitem{Podlubny:99}   Podlubny, I.:   Fractional Differential Equations. Academic Press, New York (1999)


\bibitem{Quarteroni:07} Quarteroni, A.,  Sacco, R.,  Saleri, F.:  Numerical Mathematics, 2nd ed. Springer, Berlin (2007)


\bibitem{Sakamoto:11}  Sakamoto, K., Yamamoto, M.:   Initial value/boundary value problems for fractional diffusion-wave equations and applications to some inverse problems. J. Math. Anal. Appl.  \textbf{382}, 426--447 (2011)


\bibitem{SC:20} Shi, J.K., Chen, M.H.:  Correction of high-order BDF convolution quadrature for fractional Feynman-Kac equation with L\'{e}vy flight. J. Sci. Comput., \textbf{85}:28, (2020)




\bibitem{Thomee:06} Thom\'{e}e, V.:  Galerkin Finite Element Methods for Parabolic Problems, 2nd ed.  Springer, Berlin (2006)


\bibitem{Xu:11}  Xu,  D.:  Uniform $l^1$ behaviour in a second-order difference-type method for a linear Volterra equation with completely monotonic kernel I:stability. ESAIM: M2AN. \textbf{54},   335--358 (2020)








\end{thebibliography}

\end{document}